\documentclass[10pt, letterpaper]{amsart}
\usepackage{pxfonts}
\usepackage{amsmath}
\usepackage{amssymb}
\usepackage{amsthm}
\usepackage{amscd, amsfonts}
\usepackage[dvips]{epsfig}
\usepackage{verbatim}
\usepackage{comment}
\usepackage[dvipsnames]{xcolor}

\usepackage{epsf}
\usepackage[bookmarksnumbered,pdfpagelabels=true,plainpages=false,colorlinks=true,linkcolor=black,citecolor=red,urlcolor=blue]{hyperref}

\newcommand{\spbp}{\sqrt{-1}\partial \bar{\partial}}

\newcommand{\tr}{\operatorname{tr}}
\newcommand{\dbar}{\bar{\partial}}

\newcommand{\End}{\operatorname{End}}

\theoremstyle{plain}
\newtheorem{theorem}{Theorem}[section]
\newtheorem{lemma}[theorem]{Lemma}
\newtheorem{remark}[theorem]{Remark}
\newtheorem{prop}[theorem]{Proposition}

\newtheorem{conjecture}[theorem]{Conjecture}

\theoremstyle{definition}

\numberwithin{equation}{section}

\def\XXint#1#2#3{{\setbox0=\hbox{$#1{#2#3}{\int}$}
\vcenter{\hbox{$#2#3$}}\kern-.5\wd0}}

\title[Uniqueness for the K\"ahler Yang-Mills equations]{Uniqueness for the K\"ahler-Yang-Mills equations}

\author[V. P. Pingali]{Vamsi Pritham Pingali}
\address{Department of Mathematics, Indian Institute of Science, Bangalore, India - 560012}
\email{vamsipingali@iisc.ac.in}

\author[C.-J. Yao]{Chengjian Yao}
\address{ Institute of Mathematical Sciences, ShanghaiTech University, 393 Middle Huaxia Road, Pudong,
	Shanghai, 201210 China.}
\email{yaochj@shanghaitech.edu.cn}

\begin{document}
\begin{abstract}
A formula for the $\alpha$-K-energy functional for the K\"ahler-Yang-Mills (KYM) equations is provided in this paper. Using this formula and Chen's $\epsilon$-geodesic equation on the space of K\"ahler potentials, we prove that if a solution exists to the KYM equations for a simple vector bundle on a K\"ahler manifold with discrete automorphism group, then it is unique and the $\alpha$-K-energy is bounded from below. Inspired by the study of constant scalar curvature K\"ahler metrics, we  introduce a coupled J-equation to study the $\alpha$-K energy functional.
\end{abstract}
\maketitle

\section{Introduction}
\indent In \cite{Garcia1, Garcia2} a system of partial differential equations on a K\"ahler manifold $(X,\omega)$ equipped with a  holomorphic vector bundle $E$ was introduced. It admitted a moment map interpretation akin to the case of constant scalar curvature K\"ahler metrics. Several results were proven for these equations. However, the energy functional used to define is still in a nascent stage of study. In this paper we study this functional, and prove uniqueness and lower boundedness. \\
\indent Let $(X,\omega)$ be a K\"ahler manifold  and let $(E,h_0)$ be a rank-$m$ Hermitian holomorphic vector bundle. Let $h$ be another metric on $E$ and $\omega _{\phi}=\omega+ \spbp \phi$ be another K\"ahler form. Denote by $F_h=\bar\partial \left( \partial h h^{-1}\right)$ the curvature of the Chern connection, $\mathbf{ch}_k(h) =  \frac{1}{k!}\tr \left ( \left ( \frac{\sqrt{-1}}{2\pi}F_h\right ) ^k \right )$ the curvature representation of the $k$-th Chern character class, and by $ \mathbf{bc}_k\left(h,h_0\right)$ the Bott-Chern form \cite{BC} defined by
\[
\frac{\sqrt{-1}}{2\pi} \dbar\partial \mathbf{bc}_k\left(h,h_0\right)  = \mathbf{ch}_k \left(h\right) - \mathbf{ch}_k\left(h_0\right).
\]
 In similar fashion, we will use $c_1, c_2$ to denote the curvature representation of the first and second Chern classes, i.e. 
 \[
 \det \left( I + \frac{\sqrt{-1}}{2\pi}F_{h}\right)
 = 
 1 + c_1+ c_2 +\cdots + c_m.
 \]
 The system of K\"ahler-Yang-Mills equations (KYM equations)  \cite{Garcia1, Garcia2} is the following set of PDEs (for the remainder of this paper, $\beta^{[k]}=\frac{\beta^k}{k!}$):
\begin{equation}
\begin{split}
\frac{\sqrt{-1}}{2\pi}F_{h}\wedge \omega _{\phi} ^{[n-1]} & = \lambda \omega _{\phi} ^{[n]} \otimes \mathbf{1}_E,\\
S_{\omega_{\phi}} - c
& = - 2\pi \alpha  \frac{\mathbf{ch}_2 \left( h\right)\wedge\omega_{\phi}^{[n-2]}}{\omega_{\phi}^{[n]}} + 2\pi \alpha\lambda \frac{\mathbf{ch}_1\left(h\right)\wedge\omega_{\phi}^{[n-1]}}{\omega_{\phi}^{[n]}}
\label{mainsys}
\end{split}
\end{equation}
about $\left(\omega_\phi, h\right)$. We claim (as in \cite{Garcia1, Garcia2}) that the system \eqref{mainsys} arises as the Euler-Lagrange equation of a functional that we call the $\alpha$-K-energy $\widetilde{\mathcal{K}}_\alpha$. We derive the explicit form of  $\widetilde{\mathcal{K}}_\alpha$ in Lemma \ref{newmabuchi}. A study of this functional shows the following uniqueness result over compact K\"ahler manifolds with discrete automorphism groups, which validates the idea that K\"ahler-Yang-Mills (KYM) equations may provide canonical metric pair for the pair $(X,E)$ of K\"ahler manifold and holomorphic vector bundle.
\begin{theorem}
\label{thm:uniqueness}
Let $X$ be a compact K\"ahler manifold with discrete automorphism group, and let $E$ be a simple holomorphic vector bundle over it. If $\left(\omega_{\phi_1},h_1\right)$ and $\left(\omega_{\phi_2},h_2\right)$ are both smooth solutions to \eqref{mainsys} such that $\left[\omega_{\phi_1}\right]=\left[\omega_{\phi_2}\right]$, then $\left(\omega_{\phi_1},h_1\right)=\left(\omega_{\phi_2},\kappa h_2\right)$ for some constant $\kappa>0$.
\label{thm:main}
\end{theorem}
Define $\widetilde{\mathcal{P}}:=\mathcal{H}_{\omega}\times \mathcal{H}erm^+\left(E\right),$
where $\mathcal{H}_\omega$ is the \emph{space of K\"ahler potentials} and $\mathcal{H}erm^+\left(E\right)$ is the \emph{space of Hermitian metrics on $E$}. The notation $\mathcal{H}erm\left(E\right)$ refers to the \emph{space of Hermitian forms on $E$}, i.e. the space of smooth sections of the bundle $\text{Herm}\left(E\right)$ of Hermitian forms on $E$ which are not necessarily positive definite.
\begin{theorem} 
If the system \eqref{mainsys} admits a smooth solution $\left(\omega_\phi, h\right)$, then $\widetilde{\mathcal{K}}_\alpha$ is bounded below on $\widetilde{\mathcal{P}}$.
\label{thm:existence-implies-lowerbound}
\end{theorem}
The proof of Theorem \ref{thm:main} is inspired by similar results for the Gravitating Vortex equations \cite{AllGV} using symplectic reduction by stages. One key difference is that while we consider $\epsilon$-geodesics (for the K\"ahler metric), the limiting curve of Hermitian-Einstein metrics may not exist. We instead take the limit of the $\alpha$-K-energy along these curves. An interesting point is that the ``almost" convexity of this energy along the approximating curves needs the Kobayashi-L\"ubke inequality. More generally, it is of interest to know whether a given PDE implies a Chern class inequality. Perhaps an appropriate coupling might help in conjecturing such inequalities.\\
\indent We note that the uniqueness of toric K\"ahler-Yang-Mills metrics has been obtained by Garcia-Fernandez and Lee \cite{GF-L}, by solving the toric geodesic equation directly. An explicit formula for the $\alpha$-K-energy and its asymptotic slope along geodesic rays have been obtained by Futaki and Shi \cite{FS}, as an initial attempt towards algebraic stability condition dominated by K\"ahler-Yang-Mills equations. Nonetheless, to be self-contained, we rederive the formula for the $\alpha$-K-energy in this paper.\\
\indent Finally, we introduce a coupled J-equation that couples the usual J-equation (introduced by Chen \cite{XX}) with the Hermitian-Einstein equation.  We prove that for canonically polarized manifolds, the $\alpha$-K-energy is bounded below if this equation can be solved. \\
\indent To make the paper accessible to a wider audience, we avoid the abstract moment map formalism as much as possible and stick to elementary computations.\\
\indent One of the main motivations for introducing the KYM equations is the hope of using canonical metrics to uniformize pairs of complex structures and studying a good moduli space of such pairs \cite{Garcia1}, using infinite- dimensional moment maps. The Hermitian-Einstein metric $h_\omega$ on the holomorphic vector bundle $\mathcal{E}$ with respect to the fixed background K\"ahler metric $\omega$ singles out a canonical metric on $\mathcal{E}$ with $(X,J,\omega)$ fixed. The freedom of the choice of $\omega$ leaves open the question of which background K\"ahler metric is the best to study the moduli space of complex structures on $\mathcal{E}$. The KYM equations address this problem. 

Fixed a smooth manifold $X$ and a complex vector bundle $E$ and a symplectic form $\Omega$, a coupling constant $\alpha>0$, let 
\[
\mathfrak{M}_{\Omega,\alpha}^{cx.\; str.}\left(X,E\right)
= 
\left\{
\begin{array}{ll}
\left(J_X, J_E\right)
& | J_X \text{ is an complex structure on X that is compatible with }\Omega,\\
& J_E \text{ is a holomorphic structure on E with respect to }J_X, s.t.\\
& \exists \text{ Hermitian metric }h\text{ on }\left(E, J_E\right) \; s.t.\\
&  \left(\Omega_{J_X}, h\right) \text{ is KYM metric with coupling constant }\alpha. 
\end{array}
\right\}
\]

The result in \cite{AllGV} indicates the possibility that $(J_X, J_E)$ and $(J_X, J_E')$ are not simultaneously in $\mathfrak{M}_{\Omega,\alpha}^{cx.\; str.}\left(X,E\right)$ despite  $\left(E, J_E\right)$ and $\left(E,J_E'\right)$ being \emph{slope stable} with respect to $\Omega_{J_X}$. The example in \cite{KT} shows that $(J_X, J_E)\in \mathfrak{M}_{\Omega,\alpha}^{cx.\; str.}\left(X,E\right)$ does not imply that $(X, J_X, L)$ is $K$-stable. Therefore, there may be a new ``joint stability'' condition which goes beyond $K$-stability and slope stability, yet to be fully formulated. \\

\emph{Acknowledgements}: We are grateful for Garcia-Fernandez for useful discussions. The first-named author (Pingali) acknowledges the support of the DST FIST program - 2021 [TPN - 700661] and the ANRF MATRICS grant (ANRF/ARGM/2025/000313/MTR). The authors are also grateful to ChatGPT 5.6 Sol for careful feedback on the earlier drafts of this paper. In particular, our earlier proof of Lemma 2.6 had an error, and our original proof of Lemma 2.13 was more complicated than the current one. These improvements are due to ChatGPT. The second-named author (Yao) is supported by the NSFC grant 12401071, 12371207.

\section{The alpha-K-energy}\label{sec:NewMabuchi}
We fix the coupling constant $\alpha>0$ throught this article. Define a real-valued $1$-form $\widetilde{\sigma}_\alpha$ (depending on three real parameters $\alpha$, $\widetilde\mu_1$ and $\widetilde\mu_2$) on $\widetilde{\mathcal{P}}$: at any $(\phi,h)\in \widetilde{\mathcal{P}}$, 
\begin{equation}
\label{def:moment-map-one-form}
\begin{split}
\widetilde{\sigma}_\alpha|_{\left(\phi,h\right)}\left(\psi, G\right)
& = 
\int \psi \left( \left( -S_{\omega_\phi}+c\right)\omega_\phi^{[n]} 
+ 
\widetilde\mu_1 \mathbf{ch}_1\left(h\right) \wedge \omega_\phi^{[n-1]} 
+ 
\widetilde\mu_2 \mathbf{ch}_2\left(h\right)\wedge\omega_\phi^{[n-2]}
\right)\\
& + 
\alpha\int \tr \left( \left( \frac{\sqrt{-1}}{2\pi}F_{h}\wedge \omega_\phi^{[n-1]}-\lambda \omega_\phi^{[n]}\otimes \mathbf{1}_E\right) Gh^{-1}\right)
\end{split}
\end{equation}
where $\omega_{\phi}^{[k]}:=\frac{\omega_{\phi}^k}{k!}$ (when $k\geq 0$ and $0$ otherwise) and $\left(\psi,G\right)\in T_{\left(\phi,h\right)}\widetilde{\mathcal{P}}=C^\infty\left(X,\mathbf{R}\right)\times \mathcal{H}erm\left(E\right)$.

\begin{prop}
\label{1-form-closed}
Let $\widetilde\mu_1=2\pi\alpha\lambda, \widetilde\mu_2=-2\pi\alpha.$ The expression
\[
\int_0^1dt \int \frac{\partial \phi}{\partial t}\left( \widetilde\mu_1 \mathbf{ch}_1\left(h\right)\wedge \omega_\phi^{[n-1]} 
+ 
\widetilde\mu_2 \mathbf{ch}_2\left(h\right)\wedge\omega_\phi^{[n-2]}\right)
+ 
\alpha\int_0^1dt\int \tr \left(
\left(\frac{\sqrt{-1}}{2\pi}F_h\wedge \omega_\phi^{[n-1]}-\lambda\mathbf{1}_E\otimes\omega_\phi^{[n]}\right)\frac{\partial h}{\partial t}h^{-1}
\right)
\]
is independent of the particular smooth path $\left(\phi,h\right)=\left(\phi(t),h(t)\right)$ in $\mathcal{H}_\omega\times \mathcal{H}erm^+(E)$ connecting $\left(\phi_0,h_0\right)$ and $\left(\phi_1,h_1\right)$.
\end{prop}

\begin{proof}
Let $\left(\phi,h\right)=\left(\phi(t,s),h(t,s)\right)$ be a smooth variation of the path.  The variation of the integral in the proposition with respect to $s$ is 
\begin{equation}
\begin{split}
    \int_0^1 dt \int 
    & \frac{\partial^2 \phi}{\partial s\partial t}\left( 
    \widetilde\mu_1 \mathbf{ch}_1\left(h\right)\wedge \omega_\phi^{[n-1]} 
    + 
    \widetilde\mu_2 \mathbf{ch}_2\left(h\right)\wedge \omega_\phi^{[n-2]}\right) \\
    & + 
    \frac{\partial \phi}{\partial t}
    \Bigg[ \left( \widetilde\mu_1 \frac{\sqrt{-1}}{2\pi}\bar\partial \partial \tr \left(\frac{\partial h}{\partial s}h^{-1}\right)\wedge \omega_\phi^{[n-1]} 
    + 
    \widetilde\mu_1 \mathbf{ch}_1\left(h\right)\wedge \sqrt{-1}\partial\bar\partial\left(\frac{\partial \phi}{\partial s}\right)\wedge \omega_\phi^{[n-2]}\right)\\
    & \qquad + 
    \widetilde\mu_2 \left(\frac{\sqrt{-1}}{2\pi}\right)^2\bar\partial \partial \tr \left( F_h\frac{\partial h}{\partial s}h^{-1}\right)\wedge\omega_\phi^{[n-2]}
    + 
    \widetilde\mu_2 \mathbf{ch}_2\left(h\right)\wedge \sqrt{-1}\partial\bar\partial \left(\frac{\partial \phi}{\partial s}\right)\wedge \omega_\phi^{[n-3]}\Bigg]\\
     + 
    &\alpha\int_0^1 dt \int 
     \tr \Bigg[
    \Bigg(\frac{\sqrt{-1}}{2\pi}\bar\partial \partial_h \left(\frac{\partial h}{\partial s}h^{-1}\right)\wedge\omega_\phi^{[n-1]}
    + 
    \frac{\sqrt{-1}}{2\pi}F_h \wedge \sqrt{-1}\partial\bar\partial\left(\frac{\partial\phi}{\partial s}\right)\wedge\omega_\phi^{[n-2]}\\
   & -
    \lambda \mathbf{1}_E\otimes \sqrt{-1}\partial\bar\partial\left(\frac{\partial\phi}{\partial s}\right)\wedge\omega_\phi^{[n-1]}\Bigg)\frac{\partial h}{\partial t}h^{-1}
    \Bigg]\\
     + 
    &\alpha\int_0^1 dt\int 
     \tr
    \left[ 
    \left(
    \frac{\sqrt{-1}}{2\pi}F_h\wedge \omega_\phi^{[n-1]}
    - 
    \lambda \mathbf{1}_E \otimes \omega_\phi^{[n]}\right)\frac{\partial}{\partial s}\left( \frac{\partial h}{\partial t}h^{-1}\right)
    \right],
    \end{split}
    \end{equation}
which is further simplified as
    \begin{equation}
    \begin{split}
    \int_0^1 dt \int 
    & \frac{\partial}{\partial t}\left(\frac{\partial \phi}{\partial s}\left( 
    \widetilde\mu_1 \mathbf{ch}_1\left(h\right)\wedge \omega_\phi^{[n-1]} 
    + 
    \widetilde\mu_2 \mathbf{ch}_2\left(h\right)\wedge \omega_\phi^{[n-2]}\right)\right)\\
    &- 
    \frac{\partial \phi}{\partial s}
    \Bigg[ \left( \widetilde\mu_1 \frac{\sqrt{-1}}{2\pi}\bar\partial \partial \tr \left(\frac{\partial h}{\partial t}h^{-1}\right)\wedge \omega_\phi^{[n-1]} 
    + 
    \widetilde\mu_1 \mathbf{ch}_1\left(h\right)\wedge \sqrt{-1}\partial\bar\partial\left(\frac{\partial \phi}{\partial t}\right)\wedge \omega_\phi^{[n-2]}\right)\\
    & \qquad + 
    \widetilde\mu_2 \left(\frac{\sqrt{-1}}{2\pi}\right)^2\bar\partial \partial \tr \left( F_h\frac{\partial h}{\partial t}h^{-1}\right)\wedge\omega_\phi^{[n-2]}
    + 
    \widetilde\mu_2 \mathbf{ch}_2\left(h\right)\wedge \sqrt{-1}\partial\bar\partial \left(\frac{\partial \phi}{\partial t}\right)\wedge \omega_\phi^{[n-3]}\Bigg]\\
    &+ 
    \frac{\partial \phi}{\partial t}
    \Bigg[ \left( \widetilde\mu_1 \frac{\sqrt{-1}}{2\pi}\bar\partial \partial \tr \left(\frac{\partial h}{\partial s}h^{-1}\right)\wedge \omega_\phi^{[n-1]} 
    + 
    \widetilde\mu_1 \mathbf{ch}_1\left(h\right)\wedge \sqrt{-1}\partial\bar\partial\left(\frac{\partial \phi}{\partial s}\right)\wedge \omega_\phi^{[n-2]}\right)\\
    & \qquad + 
    \widetilde\mu_2 \left(\frac{\sqrt{-1}}{2\pi}\right)^2\bar\partial \partial \tr \left( F_h\frac{\partial h}{\partial s}h^{-1}\right)\wedge\omega_\phi^{[n-2]}
    + 
    \widetilde\mu_2 \mathbf{ch}_2\left(h\right)\wedge \sqrt{-1}\partial\bar\partial \left(\frac{\partial \phi}{\partial s}\right)\wedge \omega_\phi^{[n-3]}\Bigg]\\
     + 
    \alpha\int_0^1 dt \int 
    & \tr \Bigg[
    \Bigg(\frac{\sqrt{-1}}{2\pi}\bar\partial \partial_h\left(\frac{\partial h}{\partial s}h^{-1}\right)\wedge\omega_\phi^{[n-1]}
    + 
    \frac{\sqrt{-1}}{2\pi}F_h \wedge \sqrt{-1}\partial\bar\partial\left(\frac{\partial\phi}{\partial s}\right)\wedge\omega_\phi^{[n-2]}\\
   & -
    \lambda \mathbf{1}_E\otimes \sqrt{-1}\partial\bar\partial\left(\frac{\partial\phi}{\partial s}\right)\wedge\omega_\phi^{[n-1]}\Bigg)\frac{\partial h}{\partial t}h^{-1}
    \Bigg]\\
     + 
    \alpha\int_0^1 dt\int 
    & \tr
    \left\{ 
    \left(
    \frac{\sqrt{-1}}{2\pi}F_h\wedge \omega_\phi^{[n-1]}
    - 
    \lambda \mathbf{1}_E \otimes \omega_\phi^{[n]}\right)
    \left(\frac{\partial}{\partial t}\left( \frac{\partial h}{\partial s}h^{-1}\right)
    + 
    \left[\frac{\partial h}{\partial s}h^{-1},\frac{\partial h}{\partial t}h^{-1} \right]
    \right)
    \right\}
\end{split}
\end{equation}
Denote $B_t=\frac{\partial h}{\partial t}h^{-1}, B_s=\frac{\partial h}{\partial s}h^{-1}$ and $\phi_t'=\frac{\partial \phi}{\partial t}, \phi_s'=\frac{\partial \phi}{\partial s}$ for simplicity. It follows easily that $B_t, B_s\in \Gamma\left(\End_h E\right)$. The above simplifies to 
\begin{equation}
    \begin{split}
    \int_0^1 \mathrm{d}t \int 
    &- 
    \phi_s'
    \Bigg[ \left( \widetilde\mu_1 \frac{\sqrt{-1}}{2\pi}\bar\partial \partial \tr B_t\wedge \omega_\phi^{[n-1]} 
    + 
    \widetilde\mu_1 \mathbf{ch}_1\left(h\right)\wedge \sqrt{-1}\partial\bar\partial\phi_t'\wedge \omega_\phi^{[n-2]}\right)\\
    & \qquad + 
    \widetilde\mu_2 \left(\frac{\sqrt{-1}}{2\pi}\right)^2\bar\partial \partial \tr \left( F_h B_t\right)\wedge\omega_\phi^{[n-2]}
    + 
    \widetilde\mu_2 \mathbf{ch}_2\left(h\right)\wedge \sqrt{-1}\partial\bar\partial \phi_t'\wedge \omega_\phi^{[n-3]}\Bigg]\\
    &+ 
    \phi_t'
    \Bigg[ \left( \widetilde\mu_1 \frac{\sqrt{-1}}{2\pi}\bar\partial\partial \tr B_s\wedge \omega_\phi^{[n-1]} 
    + 
    \widetilde\mu_1 \mathbf{ch}_1\left(h\right)\wedge \sqrt{-1}\partial\bar\partial\phi_s'\wedge \omega_\phi^{[n-2]}\right)\\
    & \qquad + 
    \widetilde\mu_2 \left(\frac{\sqrt{-1}}{2\pi}\right)^2\bar\partial \partial \tr \left( F_hB_s\right)\wedge\omega_\phi^{[n-2]}
    + 
    \widetilde\mu_2 \mathbf{ch}_2\left(h\right)\wedge \sqrt{-1}\partial\bar\partial \phi_s'\wedge \omega_\phi^{[n-3]}\Bigg]\\
    -\alpha \int_0^1 \mathrm{d}t\int 
    &\tr \Bigg[
    \Bigg(\frac{\sqrt{-1}}{2\pi}\bar\partial \partial_h B_t \wedge\omega_\phi^{[n-1]}
    + 
    \frac{\sqrt{-1}}{2\pi}F_h \wedge \sqrt{-1}\partial\bar\partial\phi_t'\wedge\omega_\phi^{[n-2]}
    -
    \lambda \mathbf{1}_E\otimes \sqrt{-1}\partial\bar\partial\phi_t'\wedge\omega_\phi^{[n-1]}\Bigg)B_s
    \Bigg]\\
         + 
    \alpha\int_0^1 \mathrm{d}t \int 
    & \tr \Bigg[
    \Bigg(\frac{\sqrt{-1}}{2\pi}\bar\partial \partial_h B_s \wedge\omega_\phi^{[n-1]}
    + 
    \frac{\sqrt{-1}}{2\pi}F_h \wedge \sqrt{-1}\partial\bar\partial\phi_s'\wedge\omega_\phi^{[n-2]}
    -
    \lambda \mathbf{1}_E\otimes \sqrt{-1}\partial\bar\partial\phi_s'\wedge\omega_\phi^{[n-1]}\Bigg)B_t
    \Bigg]\\
     + 
    \alpha\int_0^1 \mathrm{d}t\int 
    &
     \tr
    \left\{ 
    \left(
    \frac{\sqrt{-1}}{2\pi}F_h\wedge \omega_\phi^{[n-1]}
    - 
    \lambda \mathbf{1}_E \otimes \omega_\phi^{[n]}\right) 
    \left[B_s,B_t\right]
    \right\}
    \end{split}
    \end{equation}
which can be further arranged as
    \begin{equation}
    \begin{split}
    &
    \left(\frac{\widetilde\mu_1}{2\pi}
    -\alpha\lambda\right)\int_0^1 \mathrm{d}t \int \tr B_t \sqrt{-1}\partial\bar\partial \phi_s'\wedge \omega_\phi^{[n-1]}
    +
    \left(\frac{\widetilde\mu_2}{2\pi} +\alpha
    \right)\int_0^1 dt\int \tr\left(\frac{\sqrt{-1}}{2\pi}F_h B_t\right)\sqrt{-1}\partial\bar\partial \phi_s'\wedge \omega_\phi^{[n-2]}\\
        & -
    \left(\frac{\widetilde\mu_1}{2\pi}
    -\alpha\lambda\right)\int_0^1 \mathrm{d}t \int \tr B_s\sqrt{-1}\partial\bar\partial \phi_t'\wedge \omega_\phi^{[n-1]}
    -
    \left(\frac{\widetilde\mu_2}{2\pi} +\alpha
    \right)\int_0^1 \mathrm{d}t\int \tr\left(\frac{\sqrt{-1}}{2\pi}F_h B_s\right)\sqrt{-1}\partial\bar\partial \phi_t'\wedge \omega_\phi^{[n-2]}\\
    & + 
    \widetilde\mu_1 \left(\int_0^1 \mathrm{d}t \int \phi_t'\sqrt{-1}\partial\bar\partial \phi_s'\wedge\mathbf{ch}_1\left(h\right)\wedge \omega_\phi^{[n-2]} 
    - 
    \int_0^1 dt \int\phi_s'\sqrt{-1}\partial\bar\partial \phi_t'\wedge\mathbf{ch}_1\left(h\right)\wedge\omega_\phi^{[n-2]}\right)\\
    & + 
    \widetilde\mu_2 \left( \int_0^1 dt\int \phi_t'\mathbf{ch}_2\left(h\right)\wedge \sqrt{-1}\partial\bar\partial\phi_s'\wedge \omega_\phi^{[n-3]}
    - 
    \int_0^1 \mathrm{d}t \int \phi_s'\mathbf{ch}_2\left(h\right)\wedge \sqrt{-1}\partial\bar\partial \phi_t'\wedge\omega_\phi^{[n-3]}\right)\\
    & + 
    \frac{\alpha}{2\pi} 
    \int_0^1 \mathrm{d}t \int \tr\left(\sqrt{-1}\partial_h B_s\wedge \bar\partial B_t 
    -
    \sqrt{-1}\partial_h B_t\wedge\bar\partial B_s\right)\wedge\omega_\phi^{[n-1]}\\
    & +\alpha\int_0^1 \mathrm{d}t\int 
     \tr
    \left\{ 
    \left(
    \frac{\sqrt{-1}}{2\pi}F_h\wedge \omega_\phi^{[n-1]}
    - 
    \lambda \mathbf{1}_E \otimes \omega_\phi^{[n]}\right) 
    \left[B_s,B_t\right]
    \right\}. 
    \end{split}
\end{equation}
Noting that $B_s, B_t\in \Gamma\left(\End_h E\right)$, we see that 
\[
\tr\left(\sqrt{-1}\partial_h B_s\wedge \bar\partial B_t 
    -
    \sqrt{-1}\partial_h B_t\wedge\bar\partial B_s\right)\wedge\omega_\phi^{[n-1]}
    = 
    \left( \left\langle \partial_h B_s,\partial_h B_t\right\rangle_{h,\omega_\phi}
    -
    \left\langle \partial_h B_t,\partial_h B_s\right\rangle_{h,\omega_\phi}
    \right)
    \omega_\phi^{[n]}
\]
is  purely imaginary. On the other hand, since $C=\left[B_s, B_t\right]\in \Gamma\left(\sqrt{-1}\End_h E\right)$ and $\Theta = \Lambda_{\omega_\phi}\left(\frac{\sqrt{-1}}{2\pi} F_h\right)\in \Gamma\left(\End_h E\right)$, we have 
\[
\Theta_\alpha^{\phantom{\alpha}\beta}C_\beta^{\phantom{\beta}\alpha}
= 
-\Theta_\alpha^{\phantom{\alpha}\beta}h_{\beta\bar\gamma}\overline{C_\mu^{\phantom{\mu}\gamma}}h^{\alpha\bar\mu}
= 
-h_{\alpha\bar\beta}\overline{\Theta_\gamma^{\phantom{\gamma}\beta}}h^{\alpha\bar\mu}\overline{C_\mu^{\phantom{\mu}\gamma}}
=
-\overline{\Theta_\alpha^{\phantom{\alpha}\beta}C_\beta^{\phantom{\beta}\alpha}}. 
\]
As a consequence 
\[
\tr
    \left\{ 
    \left(
    \frac{\sqrt{-1}}{2\pi}F_h\wedge \omega_\phi^{[n-1]}
    - 
    \lambda \mathbf{1}_E \otimes \omega_\phi^{[n]}\right) 
    \left[B_s,B_t\right]
    \right\}
\]
is purely imaginary.

For the other terms, by integration-by-parts we have 
\[
\overline{\int_0^1 dt \int \phi_t'\sqrt{-1}\partial\bar\partial \phi_s'\wedge\mathbf{ch}_1\left(h\right)\wedge \omega_\phi^{[n-2]} }
    =
    \int_0^1 dt \int\phi_s'\sqrt{-1}\partial\bar\partial \phi_t'\wedge\mathbf{ch}_1\left(h\right)\wedge\omega_\phi^{[n-2]}
\]
and 
\[
 \overline{\int_0^1 dt\int \phi_t'\mathbf{ch}_2\left(h\right)\wedge \sqrt{-1}\partial\bar\partial\phi_s'\wedge \omega_\phi^{[n-3]}}
    =
    \int_0^1 dt \int \phi_s'\mathbf{ch}_2\left(h\right)\wedge \sqrt{-1}\partial\bar\partial \phi_t'\wedge\omega_\phi^{[n-3]}.
\]
Now in the case that $\widetilde\mu_1=2\pi\alpha\lambda$ and $\widetilde\mu_2=-2\pi\alpha$, we have that 
\[
\frac{d}{ds}
\left\{ \int_0^1dt \int \frac{\partial \phi}{\partial t}\left( \widetilde\mu_1 \mathbf{ch}_1\left(h\right)\wedge \omega_\phi^{[n-1]} 
+ 
\widetilde\mu_2 \mathbf{ch}_2\left(h\right)\wedge\omega_\phi^{[n-2]}\right)
+ 
\alpha\int_0^1dt\int \tr \left[
\left(\frac{\sqrt{-1}}{2\pi}F_h\wedge \omega_\phi^{[n-1]}-\lambda\mathbf{1}_E\otimes\omega_\phi^{[n]}\right)\frac{\partial h}{\partial t}h^{-1}
\right]
\right\}
\]
is purely imaginary. However, the integrand $\widetilde\mu_1 \mathbf{ch}_1\left(h\right)\wedge \omega_\phi^{[n-1]} 
+ 
\widetilde\mu_2 \mathbf{ch}_2\left(h\right)\wedge\omega_\phi^{[n-2]}$ is real. For $\Theta =\Lambda_{\omega_\phi}\left(\frac{\sqrt{-1}}{2\pi}F_h\right)\in \Gamma\left(\End_h E\right)$ and $B=\frac{\partial h}{\partial t}h^{-1}\in \Gamma\left(\End_h E\right)$, we have 
\[
\Theta_\alpha^{\phantom{\alpha}\beta}B_\beta^{\phantom{\beta}\alpha}
= 
\Theta_\alpha^{\phantom{\alpha}\beta}h_{\beta\bar\gamma}\overline{B_\lambda^{\phantom{\lambda}\gamma}}h^{\alpha\bar\lambda}
=
h_{\alpha\bar\beta}\overline{\Theta_\gamma^{\phantom{\gamma}\beta}}h^{\alpha\bar\lambda}\overline{B_\lambda^{\phantom{\lambda}\gamma}}
=
\overline{\Theta_\alpha^{\phantom{\alpha}\beta}B_\beta^{\phantom{\beta}\alpha}}. 
\]
Therefore, the integrand $\tr \left[
\left(\frac{\sqrt{-1}}{2\pi}F_h\wedge \omega_\phi^{[n-1]}-\lambda\mathbf{1}_E\otimes\omega_\phi^{[n]}\right)\frac{\partial h}{\partial t}h^{-1}
\right]$ is also real. 

We conclude that the integral is independent of the choice of the path $\left(\phi, h\right)=\left(\phi(t),h(t)\right)$ connecting $\left(\phi_0,h_0\right)$ and $\left(\phi_1,h_1\right)$. Taking into consideration of the well-known fact that
\[
-\int_0^1 dt \int \phi_t'\left(S_{\omega_{\phi_t}}-c\right)\omega_{\phi_t}^{[n]}
\]
is also independent of the path, we obtain a well-defined functional $\widetilde{\mathcal{K}}_\alpha$ on $\widetilde{\mathcal{P}}$. It is the anti-derivative of the $1$-form $\widetilde{\sigma}_\alpha$, and is called $\alpha$-K-energy.
\end{proof}
\begin{lemma}
\label{eqn:dB-dbarB}
    Let $h$ be a Hermitian metric. For $B\in \Gamma\left(\End_h E\right)$, it holds that
    \[
    \nabla^{1,0}_hB
    = 
    \partial B+\left[B, \partial hh^{-1}\right], \;
    \nabla^{0,1}B
    =
    \bar\partial B
    \]
    and
    \[
    \bar\partial B h
    = 
    h\overline{\partial_h B}, 
    \]
    i.e. 
    \[
    \bar\partial B_\alpha^{\phantom{\alpha}\beta}
    =     h_{\alpha\bar\gamma}\overline{\left(\partial_h B\right)_\lambda^{\phantom{\lambda}\gamma}}h^{\beta\bar\lambda}
    \]
    under local holomorphic frame of $E$.
\end{lemma}
\begin{proof}
    Take $\partial$ operation to the equation 
    \[  B_\alpha^{\phantom{\alpha}\gamma}h_{\gamma\beta}
    =
    h_{\alpha\bar\gamma}\overline{B_\beta^{\phantom{\beta}\gamma}}, 
    \]
    we obtain 
    \begin{align*}
        \partial B_\alpha^{\phantom{\alpha}\gamma}h_{\gamma\bar\beta}
        + 
        B_\alpha^{\phantom{\alpha}\gamma}\partial h_{\gamma\bar\beta}
        = 
        \partial h_{\alpha\bar\gamma}\overline{B_\beta^{\phantom{\beta}\gamma}}
        + 
        h_{\alpha\bar\gamma}\overline{\bar\partial B_\beta^{\phantom{\beta}\gamma}}.
    \end{align*}
    The equality in the lemma follows by noticing the notation $\left(\partial hh^{-1}\right)_\alpha^{\phantom{\alpha}\gamma}=\partial h_{\alpha\bar\beta}h^{\lambda\bar\beta}$ which is the $1$-form for the Chern connection of $h$.
\end{proof}

\begin{remark}
	Let $\mathcal{G}_{E,h}$ be the gauge group of the Hermitian vector bundle, i.e. the group of bundle automorphisms that cover the identity map on $M$ and preserves the Hermitian metric $h$. Let $\widetilde{\mathcal{G}}_E$ be the complex gauge group of $E$, i.e. the group of bundle automorphisms that cover the identity map on $M$. Obviously, $\widetilde{\mathcal{G}}_E$ is the complexification of $\mathcal{G}_{E,h}$, and $\widetilde{\mathcal{G}}_E/\mathcal{G}_{E,h}=\mathcal{H}erm_h^+$. Similarly,  $``\text{Ham}^c\left(M,\omega\right)''/\text{Ham}\left(M,\omega\right)=\mathcal{H}_\omega$ where $\text{Ham}\left(M,\omega\right)$ is the group of Hamiltonian symplectomorphisms of $(M,\omega)$.
	\end{remark}

\subsection{Second derivative of \emph{alpha-K-energy}}
Let us denote $B_t=\frac{\partial h_t}{\partial t}h_t^{-1}$ along the path of Hermitian metrics $h_t$ on $E$, then
\begin{equation}
\label{eqn:raw-second}
\begin{split}
\frac{\mathrm{d}^2\widetilde{\mathcal{K}}_\alpha}{\mathrm{d}t^2}\left( \phi_t,h_t\right)
& = 
- \int \phi_t'' \left(S_{\omega_{\phi_t}}-c\right)\omega_{\phi_t}^{[n]} 
+ \phi_t'\left( - \left(\phi_t'\right)_{ij}^{\phantom{ij}ij}+S_{\omega_{\phi_t}}^i\left(\phi_t'\right)_i + \left(\phi_t'\right)_i^{\phantom{i}i}\left(S_{\omega_{\phi_t}}-c\right)\right)\omega_{\phi_t}^{[n]}\\
& + 
\int\phi_t'' \left( \widetilde\mu_1 \mathbf{ch}_1\left(h_t\right)\wedge \omega_{\phi_t}^{[n-1]} + \widetilde\mu_2 \mathbf{ch}_2\left(h_t\right)\wedge \omega_{\phi_t}^{[n-2]}\right)\\
& + 
\int\phi_t'\left( \widetilde\mu_1\sqrt{-1}\partial\bar\partial \phi_t'\wedge \omega_{\phi_t}^{[n-2]}\wedge \mathbf{ch}_1\left(h_t\right) 
+ 
\widetilde\mu_2 \sqrt{-1}\partial\bar\partial \phi_t'\wedge \mathbf{ch}_2\left(h_t\right)\wedge \omega_{\phi_t}^{[n-3]}\right)\\
& + 
\widetilde\mu_1\int \phi_t' \tr \left( \frac{\sqrt{-1}}{2\pi}\bar\partial \partial_{h_t} B_t\wedge\omega_{\phi_t}^{[n-1]}\right)
+ 
\widetilde\mu_2\int \phi_t' \frac{\sqrt{-1}}{2\pi} \tr \left( \frac{\sqrt{-1}}{2\pi}F_{h_t}\wedge \bar\partial \partial_{h_t}B_t\wedge \omega_{\phi_t}^{[n-2]}\right)\\
& + \alpha\int \tr \left( \left( \frac{\sqrt{-1}}{2\pi}F_{h_t}\wedge \omega_{\phi_t}^{[n-1]}
-\lambda \omega_{\phi_t}^{[n]}\otimes \mathbf{1}_E\right) B_t'\right)\\
& + 
\alpha \int \tr \left( \frac{\sqrt{-1}}{2\pi}\bar\partial \partial_{h_t} B_t \wedge \omega_{\phi_t}^{[n-1]}B_t\right)\\
& + 
\alpha\int \tr \left( \left( \frac{\sqrt{-1}}{2\pi}F_{h_t}\wedge \omega_{\phi_t}^{[n-2]}\wedge\sqrt{-1}\partial\bar\partial \phi_t' 
- 
\lambda \omega_{\phi_t}^{[n-1]}\wedge \sqrt{-1}\partial\bar\partial \phi_t'\otimes \mathbf{1}_E \right)B_t\right). 
\end{split}
\end{equation}
Denoting $B=B_t$ and $A=\partial h_t h_t^{-1}$, and using the equations 
\begin{align*}
     \tr \left(F_h B\right)\wedge \partial\bar\partial \phi'
     & = 
     \partial \left( \tr \left(F_h B\right)\bar\partial \phi'\right)
     -
    \partial \tr\left(F_hB\right) \wedge\bar\partial \phi'\\
    & = 
    \partial \left( \tr \left(F_h B\right)\bar\partial \phi'\right)
    -\tr\partial_h\left(F_hB\right)\wedge\bar\partial\phi'\\
    & = 
    \partial \left( \tr \left(F_hB\right)\bar\partial \phi'\right)
    -
    \tr\left(F_h \partial_h B\right)\wedge\bar\partial \phi'
\end{align*}
and
\begin{align*}
   \tr\left( \bar\partial \partial_h B\; B\right)
   &=
   \bar\partial \left[ \left(\partial B_\alpha^{\phantom{\alpha}\beta} + B_\alpha^{\phantom{\alpha}\gamma}A_\gamma^{\phantom{\gamma}\beta}- A_\alpha^{\phantom{\alpha}\gamma}B_\gamma^{\phantom{\gamma}\beta}\right) B_\beta^{\phantom{\beta}\alpha}\right]
   + 
   \left(\partial_h B\right)_\alpha^{\phantom{\alpha}\beta}\wedge \bar\partial B_\beta^{\phantom{\beta}\alpha}\\
   & = 
   \frac{1}{2}\bar\partial \partial \tr \left(B^2\right) 
   + 
   \tr \left( \partial_h B\wedge \bar\partial B\right),
\end{align*}
we can simplify Equation \eqref{eqn:raw-second} as
\begin{equation}
\label{eqn:second-derivative-intermediate}
\begin{split}
\int 
& \left( \phi_t''-\left(\phi_t'\right)_i\left(\phi_t'\right)^i\right)\left(-S_{\omega_{\phi_t}} +c\right) \omega_{\phi_t}^{[n]} 
+ 
\int \left(\phi_t'\right)_{ij}\left(\phi_t'\right)^{ij}\omega_{\phi_t}^{[n]} \\
& 
+ \widetilde\mu_1 \int \phi_t'' \mathbf{ch}_1\left(h_t\right)\wedge \omega_{\phi_t}^{[n-1]} 
- 
\sqrt{-1}\partial\phi_t'\wedge \bar\partial \phi_t'\wedge \mathbf{ch}_1\left(h_t\right)\wedge \omega_{\phi_t}^{[n-2]}\\
& + 
\widetilde\mu_2 \int \phi_t'' \mathbf{ch}_2\left(h_t\right) \wedge\omega_{\phi_t}^{[n-2]} 
- 
\sqrt{-1}\partial\phi_t'\wedge \bar\partial \phi_t'\wedge \mathbf{ch}_2\left(h_t\right)\wedge \omega_{\phi_t}^{[n-3]}\\
& 
+ 
2\alpha \lambda \int \tr \left(\sqrt{-1}\partial_{h_t}B_t\wedge \bar\partial \phi_t'\wedge \omega_{\phi_t}^{[n-1]}\right) \\
& 
-2\alpha \int \tr \left(\frac{\sqrt{-1}}{2\pi}F_{h_t}\wedge \sqrt{-1}\partial_{h_t}B_t\wedge \bar\partial \phi_t' \wedge \omega_{\phi_t}^{[n-2]}\right)\\
& 
+ \frac{\alpha}{2\pi}\int \tr \left( \sqrt{-1}\partial_{h_t}B_t\wedge \bar\partial B_t\wedge \omega_{\phi_t}^{[n-1]}\right)\\
& + 
\alpha \int \tr \left(\left( \frac{\sqrt{-1}}{2\pi}F_{h_t}\wedge \omega_{\phi_t}^{[n-1]} - \lambda \omega_{\phi_t}^{[n]}\otimes \mathbf{1}_E\right) B_t'\right).
\end{split}
\end{equation}
Since $B$ satisfies
\begin{align*}
    \left(B_t\right)_\alpha^{\phantom{\alpha}\gamma}\left(\partial_t h_t\right)_{\gamma\bar\beta}
    = 
    \left(\partial_t h_t\right)_{\alpha\bar\gamma}\overline{\left(B_t\right)_\beta^{\phantom{\beta}\gamma}}, 
\end{align*}
we have 
\begin{align*}
    \left(B_t'\right)_\alpha^{\phantom{\alpha}\gamma}\left(h_t\right)_{\gamma\bar\beta}
    = 
    \left(h_t\right)_{\alpha\bar\gamma}\overline{\left(\partial_t B_t\right)_\beta^{\phantom{\beta}\gamma}}
\end{align*}
and therefore the term of the last line in the Equation \eqref{eqn:second-derivative-intermediate} is real. The terms of all the lines except the fourth and fifth lines are real.

\begin{lemma}	
	\label{lem:firstChern}
	\[
	\sqrt{-1}\partial\phi_t'\wedge \bar\partial \phi_t'\wedge \mathbf{ch}_1\left(h_t\right)\wedge \omega_{\phi_t}^{[n-2]}
	=
	\left|\partial \phi_t'\right|_{\omega_{\phi_t}}^2 \mathbf{ch}_1\left(h_t\right)\wedge \omega_{\phi_t}^{[n-1]}
	- 
 \tr \left( F_{h_t}\left(\nabla^{1,0}\phi_t', \nabla^{0,1}\phi_t'\right)\right)\frac{\omega_{\phi_t}^{[n]}}{2\pi }.
	\]
\end{lemma}

\begin{proof}
	Using the formula $\mathbf{ch}_1\left(h\right)= \frac{\sqrt{-1}}{2\pi} F_{\alpha\phantom{\alpha}i\bar j}^{\phantom{\alpha}\alpha}dz^i\wedge d\bar z^j$  and normal coordinates for $\omega_\phi$, i.e. $\omega_{\phi}|_p=\sqrt{-1}\delta_{ij}dz^i\wedge d\bar z^j$, we obtain the following equation.
	\begin{align*}
	LHS
	& = \frac{\left(\sqrt{-1}\right)^n}{2\pi}\sum_{i\neq j}\left( \phi'_i\phi'_{\bar i} F_{\alpha\phantom{\alpha}j\bar j}^{\phantom{\alpha}\alpha}
	-\phi'_i\phi'_{\bar j}F_{\alpha\phantom{\alpha}j\bar i}^{\phantom{\alpha}\alpha}
	\right)dz^1d\bar z^1\cdots dz^nd\bar z^n\\
	& = 
	\frac{\omega_\phi^{[n]}}{2\pi }
	\left( 
	\sum_i \phi'_i \phi'_{\bar i} \sum_j F_{\alpha\phantom{\alpha}j\bar j}^{\phantom{\alpha}\alpha}
	-
	\sum_{i,j}
	\phi'_{\bar j} \phi'_iF_{\alpha\phantom{\alpha}j\bar i}^{\phantom{\alpha}\alpha}\right)\\
	& = 
	\left|\partial \phi'\right|_{\phi}^2 \mathbf{ch}_1\left(h\right)\wedge \omega_{\phi}^{[n-1]}
	- 
	\tr \left( F_{h}\left(\nabla^{1,0}\phi',\nabla^{0,1}\phi'\right)\right) \frac{\omega_{\phi_t}^{[n]}}{2\pi }.
	\end{align*}
\end{proof}

By the definition $\mathbf{ch}_2\left(h\right)=\frac{1}{2}\tr \left(\left(\frac{\sqrt{-1}}{2\pi}F_{h}\right)^2\right)$ we can write
\[
\mathbf{ch}_2\left(h\right)
= 
\frac{\left(\sqrt{-1}\right)^2}{4}\sum_{p,q,r,s} C_{p\bar rq\bar s}dz^p\wedge d\bar z^r\wedge dz^q \wedge d\bar z^s,
\]
where
\[
C_{i\bar j k\bar l}
= 
\frac{1}{\left(2\pi\right)^2} 
\left( F_{\alpha\phantom{\beta}i\bar j}^{\phantom{\alpha}\beta}
F_{\beta\phantom{\alpha}k\bar l}^{\phantom{\beta}\alpha}
- 
F_{\alpha\phantom{\beta}k\bar j}^{\phantom{\alpha}\beta}
F_{\beta\phantom{\alpha}i\bar l}^{\phantom{\beta}\alpha}\right)
\]
satisfies the conventional ``symmetry conditions'' $C_{p\bar r q\bar s}=-C_{q\bar r p\bar s}, C_{i\bar jj\bar j}=0$ etc.

If we denote $Y\in \Gamma\left(\End_h E\right)$ to be the Hermitian-Yang-Mills endomorphism with respect to the K\"ahler metric $\omega=\sqrt{-1}g_{i\bar j}dz^i\wedge d\bar z^j$ (whose vanishing condition is precisely the Hermitian-Yang-Mills condition with respect to the K\"ahler metric $\omega$), i.e. 
\begin{align*}
   \frac{1}{2\pi} g^{i\bar j}F_{\alpha\phantom{\beta}i\bar j}^{\phantom{\alpha}\beta}
 - \lambda \delta_\alpha^{\phantom{\alpha}\beta} 
 := 
 Y_{\alpha\phantom{\beta}}^{\phantom{\alpha}\beta},
 \end{align*}
 then
\begin{lemma}
\label{lemma-HYM}
\begin{align*}
    \frac{\mathbf{ch}_2\left(h\right)\wedge \omega^{[n-2]}}{\omega^{[n]}}
    = 
    \frac{m\lambda^2}{2}
    + 
    \frac{1}{2}\left|Y\right|_h^2 
    + 
    \lambda \tr Y
    - 
    \frac{1}{8\pi^2}\left|F_h\right|_{\omega, h}^2.
  \end{align*}  
\end{lemma}

\begin{proof}
    Using normal coordinates for $\omega$ at a given point as above, we have the following:
    \begin{align*}
        \mathbf{ch}_2\left(h\right)\wedge\omega^{[n-2]}
        & = 
        \frac{1}{8\pi^2} \sum_{a,\beta,i,j,k,l}\left(F_{\alpha\phantom{\beta}i\bar j}^{\phantom{\alpha}\beta}d z^i\wedge d\bar z^j \wedge F_{\beta\phantom{\alpha}k\bar l}^{\phantom{\beta}\alpha}dz^k\wedge d\bar z^l \wedge 
        \sum_{p<q} dz^1\wedge d\bar z^1 \cdots \widehat{dz^p\wedge d\bar z^p} \cdots \widehat{dz^q\wedge d\bar z^q}\cdots dz^n\wedge d\bar z^n\right)\\
        & = 
        \frac{1}{8\pi^2} \sum_{\alpha,\beta,p<q}
        \left( F_{\alpha\phantom{\beta}p\bar p}^{\phantom{\alpha}\beta} F_{\beta\phantom{\alpha}q\bar q}^{\phantom{\beta}\alpha} 
        + 
        F_{\alpha\phantom{\beta}q\bar q}^{\phantom{\alpha}\beta}F_{\beta\phantom{\alpha}p\bar p}^{\phantom{\beta}\alpha}
        - 
        F_{\alpha\phantom{\beta}p\bar q}^{\phantom{\alpha}\beta}F_{\beta\phantom{\alpha}q\bar p}^{\phantom{\beta}\alpha}
        - 
        F_{\alpha\phantom{\beta}q\bar p}^{\phantom{\alpha}\beta}F_{\beta\phantom{\alpha}p\bar q}^{\phantom{\beta}\alpha}
        \right)\omega^{[n]}\\
        & = 
        \frac{1}{8\pi^2} \sum_{\alpha,\beta,p,q}
        \left( F_{\alpha\phantom{\beta}p\bar p}^{\phantom{\alpha}\beta} F_{\beta\phantom{\alpha}q\bar q}^{\phantom{\beta}\alpha} 
        - 
        F_{\alpha\phantom{\beta}p\bar q}^{\phantom{\alpha}\beta}F_{\beta\phantom{\alpha}q\bar p}^{\phantom{\beta}\alpha}
        \right)\omega^{[n]}\\
        & = 
        \left(\frac{1}{2}\left( Y_\alpha^{\phantom{\alpha}\beta}+\lambda \delta_\alpha^{\phantom{\alpha}\beta}\right)\left( Y_\beta^{\phantom{\beta}\alpha} + \lambda \delta_\beta^{\phantom{\beta}\alpha}\right)
        - 
        \frac{1}{8\pi^2}\left|F_h\right|_{\omega,h}^2
        \right)\omega^{[n]}\\
        & = 
        \left( \frac{1}{2}\left|Y\right|_h^2 + \lambda \tr Y + \frac{m\lambda^2}{2}
         - 
        \frac{1}{8\pi^2}\left|F_h\right|_{\omega,h}^2
        \right)\omega^{[n]}.
    \end{align*}
\end{proof}

\begin{lemma}
	\label{lem:secondChern}
	\begin{align*}
	& \sqrt{-1}\partial\phi_t'\wedge \bar\partial \phi_t'\wedge \mathbf{ch}_2\left(h_t\right)\wedge \omega_{\phi_t}^{[n-3]}\\
	& = 
	\left|\partial \phi_t'\right|_{\omega_{\phi_t}}^2 \mathbf{ch}_2\left(h_t\right)\wedge \omega_{\phi_t}^{[n-2]}
	-
	\frac{1}{2\pi}
	\tr\left(
	\frac{\sqrt{-1}}{2\pi} F_{h_t}\wedge \omega_{\phi_t}^{[n-1]}\; F_{h_t}\left(\nabla^{1,0}\phi_t',\nabla^{0,1}\phi_t'\right)
	\right) 
	+ 
	\frac{1}{\left(2\pi\right)^2}
	\left|\nabla^{0,1}\phi_t'\lrcorner  F_{h_t}\right|_{h_t,\omega_{\phi_t}}^2\omega_{\phi_t}^{[n]}.
	\end{align*}
\end{lemma}

\begin{proof}
Suppress $t$. At the point of interest $p$, if $\partial \phi' (p)=0$, then the identity holds trivially. If not, we can choose coordinates such that $\partial \phi' (p)= cdz^1$ and $\omega_{\phi}(p)=\sqrt{-1}\displaystyle \sum_{i=1}^n dz^i \wedge d\bar{z}^i$. Then the left-hand-side boils down to (using the same calculation as in Lemma \ref{lemma-HYM})
\begin{gather}
\vert c \vert^2\frac{1}{4\pi^2} \sum_{\alpha,\beta,2\leq p<q\leq n}
        \left( F_{\alpha\phantom{\beta}p\bar p}^{\phantom{\alpha}\beta} F_{\beta\phantom{\alpha}q\bar q}^{\phantom{\beta}\alpha} 
        - 
        F_{\alpha\phantom{\beta}p\bar q}^{\phantom{\alpha}\beta}F_{\beta\phantom{\alpha}q\bar p}^{\phantom{\beta}\alpha}
        \right)\omega^{[n]}.
        \label{eq:LHSsimp}
\end{gather}
From Lemma \ref{lemma-HYM} we see that the first term of the right-hand-side is
\begin{gather}
\vert c\vert^2 \frac{1}{4\pi^2} \sum_{\alpha,\beta,1\leq p<q\leq n}
        \left( F_{\alpha\phantom{\beta}p\bar p}^{\phantom{\alpha}\beta} F_{\beta\phantom{\alpha}q\bar q}^{\phantom{\beta}\alpha} 
        - 
        F_{\alpha\phantom{\beta}p\bar q}^{\phantom{\alpha}\beta}F_{\beta\phantom{\alpha}q\bar p}^{\phantom{\beta}\alpha}
        \right)\omega^{[n]}.
\label{eq:RHSfirstterm}
\end{gather}
Subtracting \ref{eq:RHSfirstterm} from \ref{eq:LHSsimp} we arrive at the following equality.
\begin{gather}
\sqrt{-1}\partial\phi_t'\wedge \bar\partial \phi_t'\wedge \mathbf{ch}_2\left(h_t\right)\wedge \omega_{\phi_t}^{[n-3]}-	\left|\partial \phi_t'\right|_{\omega_{\phi_t}}^2 \mathbf{ch}_2\left(h_t\right)\wedge \omega_{\phi_t}^{[n-2]} =  -\frac{\vert c\vert^2}{4\pi^2} \sum_{\alpha,\beta,1<q\leq n}
        \left( F_{\alpha\phantom{\beta}1\bar 1}^{\phantom{\alpha}\beta} F_{\beta\phantom{\alpha}q\bar q}^{\phantom{\beta}\alpha} 
        - 
        F_{\alpha\phantom{\beta}1\bar q}^{\phantom{\alpha}\beta}F_{\beta\phantom{\alpha}q\bar 1}^{\phantom{\beta}\alpha}
        \right)\omega^{[n]}.
        \label{eq:LHSequalsRHS}
\end{gather}
It is easy to see that at $p$, the Equation \ref{eq:LHSequalsRHS} coincides with
\begin{gather}
	-
	\frac{1}{2\pi}
	\tr\left(
	\frac{\sqrt{-1}}{2\pi} F_{h_t}\wedge \omega_{\phi_t}^{[n-1]}\; F_{h_t}\left(\nabla^{1,0}\phi_t',\nabla^{0,1}\phi_t'\right)
	\right) 
	+ 
	\frac{1}{\left(2\pi\right)^2}
	\left|\nabla^{0,1}\phi_t'\lrcorner  F_{h_t}\right|_{h_t,\omega_{\phi_t}}^2\omega_{\phi_t}^{[n]},
\end{gather}
thus completing the proof.
\end{proof}

The above computation can be used to prove the following useful lemma. 
\begin{lemma}
 For any Hermitian form $\eta=\sqrt{-1}\eta_{i\bar j}dz^i\wedge d\bar z^j$, 
 \[
 \eta\wedge \mathbf{ch}_2\left(h\right)\wedge\omega_\phi^{[n-3]}
 =
 \left(\Lambda_{\omega_\phi}\eta\right) \mathbf{ch}_2\left(h\right)\wedge\omega_\phi^{[n-2]}
 -
 \left\langle \eta, \Lambda_{\omega_\phi}\mathbf{ch}_2\left(h\right)\right\rangle_{\omega_\phi}\omega_\phi^{[n]}
 \]
 where 
 \[
 \Lambda_{\omega_\phi}\eta
 :=
 \eta_{i\bar j}g_\phi^{i\bar j},
 \;\;\;\;
 \Lambda_{\omega_\phi}\mathbf{ch}_2\left(h\right)
 :=
\sqrt{-1} C_{i\bar jk\bar l}g_\phi^{k\bar l}dz^i\wedge d\bar z^j.
 \]
\end{lemma}
We now prove another useful lemma.

\begin{lemma}
	\label{lem:mixed}
	\begin{equation*}
	\begin{split}
    & 2\alpha   \tr \left(\frac{\sqrt{-1}}{2\pi}F_{h_t}\wedge \sqrt{-1}\partial_{h_t}B_t\wedge \bar\partial \phi_t' \wedge \omega_{\phi_t}^{[n-2]}\right)\\
	& 
	\qquad =
	2\alpha \tr \left( \frac{\sqrt{-1}}{2\pi} F_{h_t}\wedge \omega_{\phi_t}^{[n-1]}\; \nabla^{1,0}\phi_t'\lrcorner \partial_{h_t} B_t\right)
	+\frac{\alpha}{\pi}
	\left\langle \partial_{h_t} B_t, \nabla^{0,1}\phi_t'\lrcorner F_{h_t}\right\rangle_{h_t,\omega_{\phi_t}}\omega_{\phi_t}^{[n]},
	\end{split}
	\end{equation*}
and
    \begin{align*}
   &\overline{2\alpha   \tr \left(\frac{\sqrt{-1}}{2\pi}F_{h_t}\wedge \sqrt{-1}\partial_{h_t}B_t\wedge \bar\partial \phi_t' \wedge \omega_{\phi_t}^{[n-2]}\right)} \\
    & \qquad =
	2\alpha \tr \left( \frac{\sqrt{-1}}{2\pi} F_{h_t}\wedge \omega_{\phi_t}^{[n-1]}\; \nabla^{0,1}\phi_t'\lrcorner\bar\partial B_t\right)
	+\frac{\alpha}{\pi}
	\left\langle \nabla^{0,1}\phi_t'\lrcorner F_{h_t}, \partial_{h_t} B_t\right\rangle_{h_t,\omega_{\phi_t}}\omega_{\phi_t}^{[n]}.
    \end{align*}

\end{lemma}

\begin{proof}
	The LHS is equal to 
	\begin{align*}
	&\frac{\alpha}{\pi}\left(\sqrt{-1}\right)^n 
	\left( F_{\alpha\phantom{\gamma}i\bar j}^{\phantom{\alpha}\gamma}dz^id\bar z^j\right) \wedge 
	\left(\nabla_k B_\gamma^{\phantom{\gamma}\alpha}dz^k\right)\wedge 
	\left( \phi'_{\bar l}d\bar z^l\right)
	\wedge 
	\sum_{p<q} dz^1d\bar z^1\cdots \widehat{dz^pd\bar z^p}\cdots \widehat{dz^qd\bar z^q}\cdots dz^n d\bar z^n\\
	&
	=
		\frac{\alpha}{\pi}\omega_\phi^{[n]}
\sum_{p<q} 
\left( F_{\alpha\phantom{\gamma}p\bar p}^{\phantom{\alpha}\gamma}\nabla_q B_\gamma^{\phantom{\gamma}\alpha}\phi'_{\bar q}
+ 
F_{\alpha\phantom{\gamma}q\bar q}^{\phantom{\alpha}\gamma}\nabla_p B_\gamma^{\phantom{\gamma}\alpha}\phi'_{\bar p}
- 
F_{\alpha\phantom{\gamma}p\bar q}^{\phantom{\alpha}\gamma}\nabla_q B_\gamma^{\phantom{\gamma}\alpha}\phi'_{\bar p}
- 
F_{\alpha\phantom{\gamma}q\bar p}^{\phantom{\alpha}\gamma}\nabla_p B_\gamma^{\phantom{\gamma}\alpha}\phi'_{\bar q}
\right)\\
& 
= 
\frac{\alpha}{\pi}\omega_\phi^{[n]} 
\left( \sum_{\alpha,\gamma}\left(\sum_p F_{\alpha\phantom{\gamma}p\bar p}^{\phantom{\alpha}\gamma}\right)\left(\sum_q \phi'_{\bar q}\nabla_q B_\gamma^{\phantom{\gamma}\alpha}\right)
- 
\sum_{p,q,\alpha,\gamma} F_{\alpha\phantom{\gamma}p\bar q}^{\phantom{\alpha}\gamma}\nabla_q B_\gamma^{\phantom{\gamma}\alpha}\phi'_{\bar p}\right). 
	\end{align*}
	
	We also recognise that 
	
	\begin{align*}
		2\alpha \tr \left( \frac{\sqrt{-1}}{2\pi} F_h\wedge \omega_\phi^{[n-1]} \; \nabla^{1,0}\phi'\lrcorner \partial_h B\right)
		& = 
		\frac{\alpha}{\pi} \omega_\phi^{[n]} 
		\sum_{\alpha,\gamma}
		\left(
		\sum_p F_{\alpha\phantom{\gamma}p\bar p}^{\phantom{\alpha}\gamma}
		\sum_q \phi'_{\bar q}\nabla_q B_\gamma^{\phantom{\gamma}\alpha}
		\right),
	\end{align*}
	and that
	\begin{align*}
	 -\sum_{p,q,\alpha,\gamma} F_{\alpha\phantom{\gamma}p\bar q}^{\phantom{\alpha}\gamma}\nabla_q B_\gamma^{\phantom{\gamma}\alpha}\phi'_{\bar p}
    & = 
     \sum_{q,\alpha,\lambda} \nabla_q B_\gamma^{\phantom{\gamma}\alpha} 
   \overline{\left( \nabla^{0,1}\phi'\lrcorner F_\lambda^{\phantom{\lambda}\beta}\right)_q}h_{\alpha\bar\beta}h^{\gamma\bar\lambda}
    = 
    \left\langle \partial_h B, \nabla^{0,1}\phi'\lrcorner F_{h}\right\rangle_{h,\omega_\phi},
	\end{align*}
where the inner product denotes the Hermitian inner product on $\Lambda^{1,0}\otimes \text{End}_h E$, i.e. $\left\langle C,D\right\rangle_{h,\omega_\phi}= C_{\alpha\phantom{\beta}i}^{\phantom{\alpha}\beta}\overline{D_{\alpha\phantom{\beta}j}^{\phantom{\alpha}\beta}}g_\phi^{i\bar j}$.	\\

The ``skew-Hermitian'' relation $\overline{F_\alpha^{\phantom{\alpha}\beta}}= - h_{\kappa\bar\alpha}F_\gamma^{\phantom{\gamma}\kappa} h^{\gamma\bar\beta}$ together with the ``Hermitian'' relation of $\partial_h B$ (Lemma \ref{eqn:dB-dbarB}) implies 
\begin{align*}
    \overline{\tr\left(\sqrt{-1}F_h\wedge \sqrt{-1}\partial_h B\wedge \bar\partial \phi'\wedge \omega_\phi^{[n-2]}\right)}
    & = 
    \sqrt{-1}^2\overline{F_\alpha^{\phantom{\alpha}\beta}}\wedge \overline{\left(\partial_h B\right)_\beta^{\phantom{\beta}\alpha}}\wedge \partial\phi'\wedge\omega_\phi^{[n-2]}\\
    & = 
    - \sqrt{-1}^2 h_{\kappa\bar\alpha}F_\gamma^{\phantom{\gamma}\kappa}h^{\gamma\bar\beta}\wedge \overline{\left(\partial_h B\right)_\beta^{\phantom{\beta}\alpha}}\wedge \partial \phi'\wedge\omega_\phi^{[n-2]}\\
    & = 
    -\sqrt{-1}^2 F_\gamma^{\phantom{\gamma}\kappa}\wedge\bar\partial B_\kappa^{\phantom{\kappa}\gamma}\wedge\partial\phi'\wedge\omega_\phi^{[n-2]}\\
    & = 
    \tr \left(\sqrt{-1}F_h\wedge \sqrt{-1}\partial\phi'\wedge \bar\partial B\wedge \omega_\phi^{[n-2]}\right).
\end{align*}
The remaining part is similar to the previous case and is omitted.
\end{proof}

Considering that $\frac{\mathrm{d}^2 \mathcal{K}_\alpha}{\mathrm{d}t^2}$ and all the lines except the fourth and fifth lines of Equation \eqref{eqn:second-derivative-intermediate} are real, the two lines can be replaced by 
\begin{align}
\label{eqn:reality}
    &\alpha\lambda \int\tr \left( 
    \nabla_{\phi_t}^{0,1}\phi_t' \lrcorner \bar\partial B_t 
    + 
    \nabla_{\phi_t}^{1,0}\phi_t'\lrcorner \partial_{h_t}B_t\right)
    \omega_{\phi_t}^{[n]}\\
    & - 
    \alpha \int \tr\left( \frac{\sqrt{-1}}{2\pi}F_{h_t} \wedge \omega_{\phi_t}^{[n-1]}\left( \nabla_{\phi_t}^{1,0}\phi_t'\lrcorner \partial_{h_t} B_t 
    +
    \nabla_{\phi_t}^{0,1}\phi_t' \lrcorner \bar\partial B_t\right)\right)
    - 
    \frac{\alpha}{2\pi}\int \left( \left\langle \partial_{h_t}B_t, \nabla_{\phi_t}^{0,1}\phi_t'\lrcorner F_{h_t}\right\rangle 
    + 
    \left\langle \nabla_{\phi_t}^{0,1}\phi_t'\lrcorner F_{h_t}, \partial_{h_t}B_t\right\rangle \right)\omega_{\phi_t}^{[n]}.\nonumber
\end{align}

Using Lemmata \ref{lem:firstChern}, \ref{lem:secondChern}, \ref{lem:mixed} and Equation \eqref{eqn:reality}, Equation \eqref{eqn:second-derivative-intermediate} can be simplified as follows.

\begin{equation}
\label{eqn:second-derivative-simplified}
\begin{split}
\int 
& \left( \phi_t''-\left(\phi_t'\right)_i\left(\phi_t'\right)^i\right)\left[ \left(-S_{\omega_{\phi_t}} +c\right) \omega_{\phi_t}^{[n]}
+
\widetilde\mu_1 \mathbf{ch}_1\left(h_t\right)\wedge \omega_{\phi_t}^{[n-1]} 
+ 
\widetilde\mu_2  \mathbf{ch}_2\left(h_t\right) \wedge\omega_{\phi_t}^{[n-2]} 
\right]
+ 
\int \left(\phi_t'\right)_{ij}\left(\phi_t'\right)^{ij}\omega_{\phi_t}^{[n]} \\
& 
+\alpha\lambda 
\int \tr \left( F_{h_t}\left(\nabla^{1,0}\phi_t', \nabla_{\phi_t}^{0,1}\phi_t'\right)\omega_{\phi_t}^{[n]}\right)\\
&
 -
\alpha \int \tr \left( \frac{\sqrt{-1}}{2\pi} F_{h_t}\wedge \omega_{\phi_t}^{[n-1]}\; F_{h_t}\left(\nabla_{\phi_t}^{1,0}\phi_t', \nabla_{\phi_t}^{0,1}\phi_t'\right)\right) 
+
\frac{\alpha}{2\pi}\int \left|\nabla_{\phi_t}^{0,1}\phi_t'\lrcorner F_{h_t}\right|_{h_t,\omega_{\phi_t}}^2 \omega_{\phi_t}^{[n]}\\
& 
+ 
\alpha \lambda \int \tr \left(\nabla_{\phi_t}^{0,1}\phi_t'\lrcorner \bar\partial B_t + \nabla_{\phi_t}^{1,0}\phi_t'\lrcorner \partial_{h_t}B_t\right)\omega_{\phi_t}^{[n]} \\
& 
-
\alpha
\int 
 \tr \left( \frac{\sqrt{-1}}{2\pi} F_{h_t}\wedge \omega_{\phi_t}^{[n-1]}\left( \nabla_{\phi_t}^{1,0}\phi_t'\lrcorner \partial_{h_t} B_t
 + \nabla_{\phi_t}^{0,1}\phi_t'\lrcorner \bar\partial B_t\right)\right)\\
 & \qquad
 -
\frac{\alpha}{2\pi}
\int
\left( \left\langle \partial_{h_t} B_t, \nabla_{\phi_t}^{0,1}\phi_t'\lrcorner F_{h_t}\right\rangle_{h_t,\omega_{\phi_t}}
+
\left\langle \nabla_{\phi_t}^{0,1}\phi_t'\lrcorner F_{h_t}, \partial_{h_t}B_t\right\rangle
\right)\omega_{\phi_t}^{[n]}
\\
& 
+ \frac{\alpha}{2\pi}\int \left|\partial_{h_t}B_t\right|_{h_t,\omega_{\phi_t}}^2\omega_{\phi_t}^{[n]}\\
& + 
\alpha \int \tr \left(\left( \frac{\sqrt{-1}}{2\pi}F_{h_t}\wedge \omega_{\phi_t}^{[n-1]} - \lambda \omega_{\phi_t}^{[n]}\otimes \mathbf{1}_E\right) B_t'\right).
\end{split}
\end{equation}

Collecting terms, we arrive at the following formula which was first derived using the moment map picture of \cite{Garcia1}. 
\begin{theorem}\label{thm:second-variation}
		The second derivative of $\widetilde{\mathcal{K}}_\alpha$ along the path $\left(\phi_t, h_t\right)$ is given by
	\begin{equation}
	\begin{split}
	& \int \left(\phi_t''-\left|\partial \phi'_t\right|_{\omega_{\phi_t}}^2\right) \left(
	\left( -S_{\omega_{\phi_t}} +c\right)\omega_{\phi_t}^{[n]} 
	+ 
	2\pi\alpha\lambda \mathbf{ch}_1\left(h_t\right)\wedge\omega_{\phi_t}^{[n-1]}
 -
 2\pi \alpha \mathbf{ch}_2\left(h_t\right)\wedge \omega_{\phi_t}^{[n-2]}
	\right)\\
	& + 
	\alpha \int \tr \left(\left( \frac{\sqrt{-1}}{2\pi}F_{h_t}\wedge \omega_{\phi_t}^{[n-1]} - \lambda \omega_{\phi_t}^{[n]}\otimes \mathbf{1}_E\right) \left(B_t'
    -\nabla_{\phi_t}^{1,0}\phi_t'\lrcorner \partial_{h_t} B_t 
    -\nabla_{\phi_t}^{0,1}\phi_t'\lrcorner \bar\partial B_t
    - 
	F_{h_t}\left(\nabla_{\phi_t}^{1,0}\phi_t', \nabla_{\phi_t}^{0,1}\phi_t'\right)\right)
	\right)\\
	&\qquad + 
	\int \left|\bar\partial\left(\nabla_{\phi_t}^{1,0}\phi_t'\right)\right|_{\omega_{\phi_t}}^2\omega_{\phi_t}^{[n]} 
	+
	\frac{\alpha}{2\pi}\int\left| \partial_{h_t} B_t-\nabla_{\phi_t}^{0,1}\phi_t'\lrcorner F_{h_t}\right|_{h_t,\omega_{\phi_t}}^2\omega_{\phi_t}^{[n]},
	\end{split}
	\end{equation}
    where $\phi_t'=\frac{\partial \phi_t}{\partial t}$ and $B_t=\frac{\partial h_t}{\partial t}h_t^{-1}$.
\end{theorem}
Notice that in dealing with the term in the second term to last of \eqref{eqn:second-derivative-simplified} we apply Lemma \ref{eqn:dB-dbarB} to the endomorphism $B_t$ and the Hermitian metric $h_t$. More precisely, we use the following identity. 
\[
\tr \left(\sqrt{-1}\partial_{h_t} B_t\wedge \bar\partial B_t\wedge \omega_{\phi_t}^{[n-1]} \right)
=
g_{\phi_t}^{i\bar j}\nabla_i \left(B_t\right)_\alpha^{\phantom{\alpha}\beta} \left(h_t\right)_{\beta\bar\gamma}
\overline{\nabla_j \left(B_t\right)_\lambda^{\phantom{\lambda}\gamma}}\left(h_t\right)^{\alpha\bar\lambda}\omega_{\phi_t}^{[n]}
= 
\left\langle\partial_{h_t} B_t, \partial_{h_t} B_t\right\rangle_{h_t,\omega_{\phi_t}} \omega_{\phi_t}^{[n]}.
\]

\subsection{The \emph{reduced alpha-K-energy} and \emph{deformed cscK equation}}\label{sec:Newgeo}
Akin to the relationship between the second derivative of the usual $K$-energy and the geodesic equation in the space of K\"ahler potentials, the second derivative formula for the $\alpha$-$K$-energy leads to the following the ``coupled'' geodesic equation\footnote{All of the three terms in the second equation belong to $\Gamma\left(\End_{h_t}E\right)$ for any $t$.} in $\widetilde{\mathcal{P}}=\mathcal{H}_\omega\times \mathcal{H}erm^+\left(E\right)$ (first defined in \cite{Garcia1})
for the path $\left(\phi_t, h_t\right)$: 
\begin{align}
\left\{
\begin{array}{rl}
 \frac{\partial^2 \phi_t}{\partial t^2}-\left|\nabla_{\phi_t}^{1,0}\left(\frac{\partial \phi_t}{\partial t}\right)\right|_{\omega_{\phi_t}}^2
    & =0, \\
     \frac{\partial}{\partial t}\left(\frac{\partial h_t}{\partial t} h_t^{-1}\right)
    - \nabla^{1,0}_{\phi_t}\left(\frac{\partial \phi_t}{\partial t}\right)\lrcorner \partial_{h_t} \left(\frac{\partial h_t}{\partial t} h_t^{-1}\right)
    - \nabla_{\phi_t}^{0,1}\left(\frac{\partial \phi_t}{\partial t}\right)\lrcorner \bar\partial \left(\frac{\partial h_t}{\partial t}h_t^{-1}\right)
    - 
	F_{h_t}\left(\nabla^{1,0}_{\phi_t}\left(\frac{\partial \phi_t}{\partial t}\right), \nabla^{0,1}_{\phi_t}\left(\frac{\partial \phi_t}{\partial t}\right)\right)
    & = 0.
\end{array}
    \right.
\end{align}

In terms of the Hermitian metric (instead of endomorphism), the geodesic equation can be written as 
\begin{align}
\left\{
\begin{array}{rl}
\ddot\phi - \frac{1}{2}\left|\nabla_\phi\dot\phi\right|_\phi^2 
& = 0,\\
\ddot h - \dot h h^{-1}\dot h 
- 
\nabla_\phi\dot\phi\lrcorner \nabla_h \dot h
- 
\frac{\sqrt{-1}}{2}F_h\left(\nabla_\phi\dot\phi, J\nabla_\phi\dot \phi\right)
& = 0,
\end{array}
\right.
\end{align}
where 
\begin{align*}
\nabla_\phi \dot\phi
& = 
\nabla^{1,0}_\phi \dot\phi + \nabla^{0,1}_\phi \dot\phi, \\
\nabla_h \dot h
& =
 \left(\partial \dot h - \partial h h^{-1}\dot h\right) 
+ 
 \left(\bar\partial \dot h - \dot h h^{-1}\bar\partial h\right)
   = \nabla_h'\dot h + \nabla_h'' \dot h.
\end{align*}

Due to the potential degeneracy of the ``K\"ahler metrics'' along the geodesic in the space of K\"ahler potentials (for instance, it is difficult to make sense of $\nabla_\phi^{1,0}\phi$), we need to rewrite the equations in a way similar to the well-known Donaldson-Semmes mechanism of transforming geodesic equation to \emph{Homogeneous Complex Monge-Amp\`ere equation}. Let $M= X\times \Sigma$ where the cylinder $\Sigma= [0,1]\times S^1$ is equipped with the coordinate $w=t+\sqrt{-1}\theta$. Let $\widetilde{E}$ be the pull-back bundle of $E$ under the natural projection map $M\to X$, and $\widetilde{h}|_{(p,w)}=h_t|_p$ (with $t=\text{Re }w$) be the natural pull-back Hermitian metric, and $\left\{\widetilde{\mathbf e}_\alpha\right\}_{\alpha=1,2,\cdots, m}$ be the natural frame obtained by pulling back $\left\{\mathbf{e}_\alpha\right\}_{\alpha=1,2,\cdots, m}$. Let $\Phi(p,w)=\phi_t\left(p\right)$ with $p\in X$ and $t=\text{Re }w$, and $\Omega = \pi^*\omega_{\phi_t} + \sqrt{-1}\partial_X\bar\partial_X \Phi$ be a closed real $(1,1)$-form on $M$. We have 
\begin{align*}
    \left( F_{\widetilde h} \right)_{\alpha\bar\beta}
    & = 
    -\partial_M \bar\partial_M \widetilde{h}_{\alpha\bar\beta}
    + 
    \partial_M \widetilde{h}_{\alpha\bar\gamma}\widetilde{h}^{\lambda\bar\gamma}\wedge\bar\partial_M \widetilde{h}_{\lambda\bar\beta}\\
    & = 
    \left(F_{h_t}\right)_{\alpha\bar\beta}
    - 
    \frac{1}{2}\left(\nabla_{h_t}' \dot h_t\right)_{\alpha\bar\beta}\wedge\mathrm{d}\bar w
    - 
    \frac{1}{2}\mathrm{d}w\wedge \left(\nabla_{h_t}''\dot h_t\right)_{\alpha\bar\beta}
    - 
    \frac{1}{4}\left(\ddot h_t- \dot h_t h_t^{-1}\dot h_t\right)_{\alpha\bar\beta}\mathrm{d}w\wedge\mathrm{d}\bar w, \\
    \Omega 
    & = 
    \pi^*\omega_{\phi_t}
    + 
    \frac{\sqrt{-1}}{2}\left(\mathrm{d}w\wedge \bar\partial\dot\phi_t + \partial \dot\phi_t\wedge \mathrm{d}\bar w\right) 
    + 
    \frac{\sqrt{-1}}{4}\ddot\phi_t \mathrm{d}w\wedge \mathrm{d}\bar w.
\end{align*}
Therefore, 
\begin{align*}
    \Omega^{[n]}
    & = 
    \omega_{\phi_t}^{[n]}
    + 
    \frac{1}{2}\omega_{\phi_t}^{[n-1]}\wedge \left(\sqrt{-1}\partial\dot\phi_t\wedge \mathrm{d}\bar w + \sqrt{-1}\mathrm{d}w\wedge \bar\partial \dot\phi_t\right)\\
    &\qquad \qquad
    + 
    \frac{1}{4}\left(\ddot\phi_t \omega_{\phi_t}^{[n-1]}  - \omega_{\phi_t}^{[n-2]}\wedge \sqrt{-1}\partial\dot\phi_t\wedge \bar\partial \dot\phi_t\right)\wedge \sqrt{-1}\mathrm{d}w\wedge \mathrm{d}\bar w,\\
    \Omega^{[n+1]}
    & = 
    \frac{1}{4}\left(\ddot\phi_t \omega_{\phi_t}^{[n]}
    - 
    \sqrt{-1}\partial\dot\phi_t\wedge \bar\partial \dot \phi_t\wedge\omega_{\phi_t}^{[n-1]}\right)\wedge \sqrt{-1}\mathrm{d}w\wedge\mathrm{d}\bar w,
\end{align*}
and 
\begin{align}
    \left(\sqrt{-1}F_{\widetilde h}\right)_{\alpha\bar\beta}\wedge\Omega^{[n]}
     = 
     &- \frac{1}{4}\left(\ddot h_t - \dot h_t h_t^{-1}\dot h_t\right)_{\alpha\bar\beta}\omega_{\phi_t}^{[n]}\wedge\sqrt{-1}\mathrm{d}w\wedge\mathrm{d}\bar w \nonumber\\
     & \quad\;\; + 
    \frac{1}{4}\left(\sqrt{-1}F_{h_t}\right)_{\alpha\bar\beta}\wedge \left(\ddot\phi_t \omega_{\phi_t}^{[n-1]}- \sqrt{-1}\partial\dot\phi_t\wedge \bar\partial \dot\phi_t\wedge\omega_{\phi_t}^{[n-2]}\right)\wedge \sqrt{-1}\mathrm{d}w\wedge \mathrm{d}\bar w \nonumber\\
     &\quad\qquad +
    \frac{1}{4}\left( 
    \sqrt{-1}\partial\dot\phi_t \wedge \left(\nabla_{h_t}''\dot h_t\right)_{\alpha\bar\beta}
    + 
    \sqrt{-1}\left(\nabla_{h_t}'\dot h_t\right)_{\alpha\bar\beta}\wedge \bar\partial \dot\phi_t\right)\wedge \omega_{\phi_t}^{[n-1]}\wedge \sqrt{-1}\mathrm{d}w\wedge\mathrm{d}\bar w \nonumber\\
     = 
     &-
     \frac{1}{4}\left[  
    \ddot h_t -\dot h_th_t^{-1}\dot h_t
    - 
    \nabla_{\phi_t} \dot\phi_t\lrcorner \nabla_{h_t}\dot h_t
    - 
    F_{h_t}\left(\nabla^{1,0}_{\phi_t}\dot\phi_t, \nabla^{0,1}_{\phi_t}\dot\phi_t\right)
    \right]_{\alpha\bar\beta}\omega_{\phi_t}^{[n]}\wedge \sqrt{-1}\mathrm{d}w\wedge \mathrm{d}\bar w\nonumber\\
    &\quad\;\; +  \left(\Lambda_{\omega_{\phi_t}}\sqrt{-1}F_{h_t}\right)_{\alpha\bar\beta}\Omega^{[n+1]}.
\end{align}
It is in this sense that the ``coupled'' geodesic equation could be rewritten as 
\begin{align}
\left\{
\begin{array}{rl}
\Omega^{[n+1]}
& = 0, \\
\sqrt{-1}F_{\widetilde h}\wedge \Omega^{[n]}
& =0.
\end{array}
\right.
\end{align}
In terms of background K\"ahler metric $\Omega_0=\pi^*\omega_0 + \sqrt{-1}\mathrm{d}w\wedge\mathrm{d}\bar w$, we can write $\Omega=\Omega_0+\sqrt{-1}\partial_M\bar\partial_M \Psi$ with $\Psi=\Phi- 2t\left(t-1\right)$, this is a coupled system of \emph{Homogeneous Complex Monge-Amp\`ere} and \emph{Degenerate Hermitian-Yang-Mills} equations on the manifold $M$ with boundary $X\times\{0,1\}$, together with the boundary conditions
\begin{align*}
\left\{
\begin{array}{lll}
\Psi|_{X\times \{k\}\times S^1}
& = \phi_k, \; &k=0,1;\\
\widetilde h|_{X\times \{k\}\times S^1}
& = h_k, \; &k=0,1.
\end{array}
\right.
\end{align*}
And we stay in the realm of positive current, i.e. $\Omega_0+\sqrt{-1}\partial_M\bar\partial_M \Psi\geq 0$. One na\"ive approach is to solve the ``$\epsilon$-approximated coupled geodesic equation'': 
\begin{align}
\left\{
\begin{array}{rl}
\left(\Omega_0+\sqrt{-1}\partial\bar\partial \Psi_\varepsilon\right)^{[n+1]}
& = \varepsilon \Omega_0^{[n+1]}, \\
\sqrt{-1}F_{\widetilde h_\varepsilon}\wedge \left(\Omega_0+\sqrt{-1}\partial\bar\partial \Psi_\varepsilon\right)^{[n]}
& =0.
\end{array}
\right.
\end{align}
The first equation could be solved easily, and the second equation can also be solved relatively easily. The difficulty is then the ``a priori estimates'' for $\widetilde h_\varepsilon$ considering the kind of \emph{optimal} $C^{1,\bar 1}$ estimate about $\Psi_\epsilon$ in the first equation by Chen \cite{XX}.

Instead of dealing with this ``coupled'' geodesic equation (as which has been recently used in \cite{GF-L} to obtain uniqueness in the toric setting), we use the approach of \emph{Symplectic Reduction by Stages}, developed in \cite{AllGV} for the gravitating vortices on Riemann surfaces, and study the \emph{reduced $\alpha$-$K$-energy}.

Consider the embedded submanifold
\[\mathcal{P}_{HE}:=\left\{ \left(\phi, h\right)\in \widetilde{\mathcal{P}}=\mathcal{H}_\omega\times \mathcal{H}erm^+\left(E\right)| h\text{ is Hermitian-Einstein w.r.t. }\omega_\phi\right\}\]
and the restriction of $\widetilde{\mathcal{K}}_\alpha$ to $\mathcal{P}_{HE}$. 

If we assume $E$ is \emph{slope polystable} with respect to the polarization $[\omega]$, Donaldson-Uhlenbeck-Yau Theorem \cite[Theorem 6.10.19]{Ko} establishes the existence of Hermitian-Einstein metrics for any $\omega_\phi\in \left[\omega\right]$. The Hermitian-Einstein metric is never uniquely determined since a positive multiple of a Hermitian-Einstein metric is also Hermitian-Einstein. In general, we have the uniqueness of the Hermitian-Einstein connections though \cite[Proposition 6.3.37]{Ko}. If we assume $E$ is \emph{simple}, then for each $\omega_\phi$, there exists at most one Hermitian-Einstein metric up to an overall constant scaling.

With the above in mind, we can consider the embedding
\[
\mathbf{h}: \mathcal{H}_\omega\rightarrow \widetilde{\mathcal{P}}
\]
and let $\mathcal{K}_\alpha:=\mathbf{h}^*\widetilde{\mathcal{K}}_\alpha$. By the naturality, $\mathrm{d}\mathcal{K}_\alpha=\mathbf{h}^*\left(\mathrm{d}\widetilde{\mathcal{K}}_\alpha\right) =\mathbf{h}^*\widetilde{\sigma}_\alpha=\sigma_\alpha$, where 
\[
\sigma_\alpha|_\phi \left(\psi\right)
: = 
\int \psi \left( \left( -S_{\omega_\phi}+c\right)\omega_\phi^{[n]} 
+ 
2\pi\alpha \lambda \mathbf{ch}_1\left(h^\phi\right) \wedge \omega_\phi^{[n-1]} 
- 2\pi\alpha  \mathbf{ch}_2\left(h^\phi\right)\wedge\omega_\phi^{[n-2]}
\right)
\] 
is a $1$-form on $\mathcal{H}_\omega$ and $\mathbf{ch}_i\left(h^\phi\right)$ are the unique Chern forms of corresponding degree (recall that while the Hermitian-Einstein metrics are not uniquely determined, the connection and the curvature are uniquely determined \cite[Proposition 6.3.37]{Ko})  of a Hermitian-Einstein metric $h^\phi$ for the background K\"ahler metric $\omega_\phi$.  Moreover, we have 
\begin{equation}
\mathcal{K}_\alpha\left(\phi\right)
= 
\inf_{h\in \mathcal{H}erm^+\left(E\right)}\widetilde{\mathcal{K}}_\alpha\left(\phi, h\right)
= 
\widetilde{\mathcal{K}}_\alpha\left(\phi, h^\phi\right),
\end{equation}
which is called the \emph{reduced $\alpha$-K-energy}. This is the analogy of finite dimensional ``reduced Kempf-Ness functional'' in the \emph{reduced moment map} picture via the process of \emph{symplectic reduction by stages} \cite{AllGV}. The critical point of $\mathcal{K}_\alpha$ is a K\"ahler potential $\phi$ satisfying the ``non-local'' equation 
\begin{equation}
\label{eqn:deformedCSCK}
\left( -S_{\omega_\phi}+c\right)\omega_\phi^{[n]} 
+ 
2\pi\alpha \lambda \mathbf{ch}_1\left(h^\phi\right) \wedge \omega_\phi^{[n-1]} 
-
2\pi\alpha \mathbf{ch}_2\left(h^\phi\right)\wedge \omega_\phi^{[n-2]}
=
0,
\end{equation}
which by Lemma \ref{lemma-HYM} is equivalent to\footnote{Notice that the overall scaling of $h^\phi$ does not affect $\left|F_{h^\phi}\right|^2_{h^\phi,\omega_{\phi}}$ since the curvature belongs to the bundle of endomorphisms.}
\begin{equation}
S_{\omega_\phi}
= 
c+ \pi \alpha m\lambda^2 
+ \frac{\alpha}{4\pi}\left|F_{h^\phi}\right|^2_{h^\phi,\omega_{\phi}}.
\end{equation}

This can be viewed as a type of ``deformed cscK equation'', which takes account of the back-reaction of the gauge field to the usual cscK metric equation. \\

\begin{remark}
	The deformed cscK metric of line bundles can be written as a second order PDE system for a triple of functions $\left(\phi, G,f\right)$ where 
	\begin{align*}
	& \Delta_{\omega_\phi}G
	 = 
	\tr_{\omega_\phi} \text{Ric }\omega - \widetilde{c} - \pi\alpha \left|\Theta+i\partial\bar\partial f\right|_{\omega_\phi}^2 \\
	&\left(\Theta+i\partial\bar\partial f\right) \wedge \omega_\phi^{[n-1]} 
	 = \lambda \omega_\phi^{[n]}\\
	& \omega_\phi^{[n]}
	 = e^G \omega^{[n]}
	\end{align*}
	and $\omega$ is fixed background K\"ahler metric and $\Theta$ is a fixed background $(1,1)$-form. 
\end{remark}

\indent To study the lower boundedness and convexity of the \emph{reduced $\alpha$-K-energy}, we look at the following system
\begin{equation}
\begin{split}
& \phi''_t-\vert \nabla^{1,0} \phi_t'\vert_{\omega_{\phi_t}}^2=0, \\
&\frac{\sqrt{-1}}{2\pi}F_{h_t}\wedge \omega_{\phi_t}^{[n-1]}=\lambda\mathbf{1}_E\otimes \omega_{\phi_t}^{[n]}.
\end{split}
\label{eq:geodesics}
\end{equation}
Unfortunately, the first equation in \eqref{eq:geodesics} does not have smooth solutions in general, which inevitably leads to Hermitian-Einstein equations with non-smooth or even degenerate background K\"ahler metrics.  To remedy this problem, we consider the system of Chen's $\epsilon$-geodesics and the associated Hermitian-Einstein equations:
\begin{equation}
\begin{split}
\left( \phi_{t,\epsilon}''-\vert \nabla^{1,0} \phi_{t, \epsilon}'\vert_{\omega_{\phi_{t,\epsilon}}}^2\right)\omega_{\phi_{t, \epsilon}}^{[n]}
& =\epsilon \omega^{[n]},\\
\frac{\sqrt{-1}}{2\pi}F_{h^{\phi_{t,\epsilon}}}\wedge \omega_{\phi_{t,\epsilon}}^{[n-1]}
& =\lambda\mathbf{1}_E\otimes \omega_{\phi_{t,\epsilon}}^{[n]}
\end{split}
\label{eq:epsilongeodesics}
\end{equation}
for $t\in [0,1]$ and $\epsilon\in (0,1]$. Given any fixed smooth K\"ahler potentials as boundary condition, the first equation in \eqref{eq:epsilongeodesics} can be solved with uniform $C^{1,\bar 1}$ estimates (in space and time) independent of $\epsilon$ by Chen \cite{XX1}. It seems very difficult (if at all possible) to find similar estimates for $h^{\phi_{t,\epsilon}}$. The corresponding estimates for Gravitating Vortices are established in \cite{AllGV}, enabling us to solve the weak vortices equations (i.e. vortices equation with background metric being weak K\"ahler metrics) on Riemann surfaces. In dimension $n\geq 2$ case, instead of pursuing possible $C^{1,\bar 1}$ estimates for $h^{\phi_{t,\epsilon}}$ and dealing with Hermitian-Einstein equations with weak background K\"ahler metric, we will study the behaviour of the \emph{reduced $\alpha$-K-energy} along the $\epsilon$-geodesics $\left\{\phi_{t,\epsilon}\right\}_{t\in [0,1]}$. We shall then take a limit of the modified part (coming from the terms containing $\mathbf{ch}_1$ and $\mathbf{ch}_2$ in \eqref{eqn:deformedCSCK}) as $\epsilon \to 0$,  while evaluating the \emph{$K$-energy} part directly along the $C^{1,\bar 1}$-geodesic. In other words, we decompose the functional as follows. 
\begin{equation}
    \mathcal{K}_\alpha 
    = 
    \mathcal{K}
    +
    \mathcal{M}_\alpha: \mathcal{H}_\omega\rightarrow \mathbf{R},
\end{equation}
where

\begin{align*}
    \mathcal{K}\left(\phi\right) 
   & = 
    \int_0^1 \mathrm{d}t \int \phi_t'\left( -S_{\omega_{\phi_t}}+\underline{S}\right)\omega_{\phi_t}^{[n]}, \\
    \mathcal{M}_\alpha\left(\phi\right)
    & = 
    \int_0^1 \mathrm{d}t \int \phi_t'\left(
    2\pi\alpha \lambda \mathbf{ch}_1\left(h^{\phi_t}\right) \wedge \omega_{\phi_t}^{[n-1]} 
    -
   2\pi\alpha \mathbf{ch}_2\left(h^{\phi_t}\right)\wedge \omega_{\phi_t}^{[n-2]}\right)
   + \int_0^1 \mathrm{d}t\int \phi_t'\left(c-\underline{S}\right)\omega_{\phi_t}^{[n]}
\end{align*}
for any smooth curve $\phi_t$ connecting the fixed background K\"ahler potential $\phi_0$ with the K\"ahler potential $\phi$.

Denote $g_\epsilon(t)= \mathcal{M}_\alpha \left(\phi_{t,\epsilon}\right)$.
	\begin{theorem}
		\[
		g_\epsilon''(t)
		\geq 
		\left(\pi\alpha m\left(1+\frac{1}{n}\right)\lambda^2+c-\underline{S}\right)\text{Vol}_\omega\cdot \epsilon, \;\; \forall t\in [0,1],\; \epsilon>0.
		\]
	\end{theorem}
	
	\begin{proof}
		\begin{equation}
		\begin{split}
		g_\epsilon''(t)
		&
		= \int \left(\phi_{t,\epsilon}''-\left|\partial \phi'_{t,\epsilon}\right|_{\omega_{\phi_{t,\epsilon}}}^2\right)\omega_{\phi_{t,\epsilon}}^{[n]}  \left( 
		2\pi\alpha\lambda \frac{\mathbf{ch}_1\left(h^{\phi_{t,\epsilon}}\right)\wedge  \omega_{\phi_{t,\epsilon}}^{[n-1]}}{\omega_{\phi_{t,\epsilon}}^{[n]}}
		-
		2\pi \alpha \frac{\mathbf{ch}_2\left(h^{\phi_{t,\epsilon}}\right)\wedge\omega_{\phi_t}^{[n-2]}}{\omega_{\phi_{t,\epsilon}}^{[n]}}
	+c-\underline{S}
        \right)\\
		&\qquad
		+
		\frac{\alpha}{2\pi}\left|\left| \partial_{h^{\phi_{t,\epsilon}}} B_{t,\epsilon}-\nabla^{0,1}\phi_{t,\epsilon}'\lrcorner F_{h^{\phi_{t,\epsilon}}}\right|\right|_{h^{\phi_{t,\epsilon}},\omega_{\phi_{t,\epsilon}}}^2\\
		& 
		= 
		2\pi\alpha\epsilon \int  \omega^{[n]} 
		\left( 
		\lambda \cdot m\lambda 
		- 
		\frac{\mathbf{ch}_2\left(h^{\phi_{t,\epsilon}}\right)\wedge \omega_{\phi_{t,\epsilon}}^{[n-2]}}{\omega_{\phi_{t,\epsilon}}^{[n]}}
        +\frac{c-\underline{S}}{2\pi \alpha}
		\right)
		+
		\frac{\alpha}{2\pi}\left|\left| \partial_{h^{\phi_{t,\epsilon}}} B_{t,\epsilon}-\nabla^{0,1}\phi_{t,\epsilon}'\lrcorner F_{h^{\phi_{t,\epsilon}}}\right|\right|_{h^{\phi_{t,\epsilon}},\omega_{\phi_{t,\epsilon}}}^2\\
		&
		\geq 
		2\pi\alpha\epsilon \int \omega^{[n]} 
		\left(m\lambda^2 
		-
		\frac{1}{2m}\frac{\mathbf{ch}_1^2\left(h^{\phi_{t,\epsilon}}\right)\wedge \omega_{\phi_{t,\epsilon}}^{[n-2]}}{\omega_{\phi_{t,\epsilon}}^{[n]}}
 +\frac{c-\underline{S}}{2\pi \alpha}
        \right) \\
		& 
		= 
		2\pi\alpha\epsilon \int
		\left(m\lambda^2 
		-
		\frac{1}{2m}
		\left(
		\left| \left\langle \mathbf{ch}_1\left(h^{\phi_{t,\epsilon}}\right), \omega_{\phi_{t,\epsilon}}\right\rangle_{\omega_{\phi_{t,\epsilon}}}\right|^2 
		- 
		\left| \mathbf{ch}_1\left(h^{\phi_{t,\epsilon}}\right) \right|_{\omega_{\phi_{t,\epsilon}}}^2
		\right)
 +\frac{c-\underline{S}}{2\pi \alpha}
		\right) \omega^{[n]}        \\
		& 
		= 
		\pi\alpha\epsilon 
		\int \left( m\lambda^2 + \frac{1}{m}\left| \mathbf{ch}_1\left(h^{\phi_{t,\epsilon}}\right) \right|_{\omega_{\phi_{t,\epsilon}}}^2
         +\frac{c-\underline{S}}{\pi \alpha}
		\right)
		\omega^{[n]}
		\end{split}
		\end{equation}	
		where we have used 
		\[
		\left\langle \mathbf{ch}_1\left(h^{\phi_{t,\epsilon}}\right), \omega_{\phi_{t,\epsilon}}\right\rangle_{\omega_{\phi_{t,\epsilon}}}
		=
		m\lambda
		\]
		by taking trace of the Hermitian-Einstein condition
		and the Kobayashi-L\"ubke inequality \cite{Ko}
		\begin{equation}
        \label{ineq:KL-inequality}
		\left(2m \mathbf{ch}_2\left(h^{\phi_{t,\epsilon}}\right)-\mathbf{ch}_1^2\left(h^{\phi_{t,\epsilon}}\right)\right)\wedge \omega_{\phi_{t,\epsilon}}^{[n-2]}
		\leq 
		0
		\end{equation}
		for Hermitian-Einstein metrics. The final inequality about $g_\epsilon''$ follows. 
	\end{proof}

\begin{theorem}There exists $C>0$ independent of $t\in [0,1]$ and $\epsilon\in (0,1]$, such that $g_\epsilon''(t)\leq C, \;\; \forall t\in [0,1], \epsilon\in (0,1]$. 
\end{theorem}

\begin{proof}
	The second derivative along the curve of Hermitian-Einstein metrics is 
	\begin{equation}
	\label{eqn:second-upper-bound}
	\begin{split}
	g_\epsilon''(t)
	& = 
	(2\pi \alpha\lambda^2 m +c-\underline{S}) \int \left( 
	\phi_t'' \omega_{\phi_t}^{[n]} 
	- 
	\sqrt{-1}\partial\phi_t'\wedge \bar\partial\phi_t'\wedge \omega_{\phi_t}^{[n-1]}	\right)
	- 
	2\pi\alpha \int \phi_t'' \mathbf{ch}_2\left(h_t\right)\wedge \omega_{\phi_t}^{[n-2]}\\
	& \qquad
	- \alpha\int \phi_t' \tr 
	\left( \frac{\sqrt{-1}}{2\pi}F_{h_t}\wedge 
	\sqrt{-1}\bar\partial \partial_{h_t} B_t\wedge \omega_{\phi_t}^{[n-2]}\right)\\
	& \qquad\qquad
	- 
	2\pi \alpha \int\phi_t'\sqrt{-1}\partial\bar\partial \phi_t'\wedge \mathbf{ch}_2\left(h_t\right)\wedge \omega_{\phi_t}^{[n-3]}\\
	& = 
(2\pi \alpha\lambda^2 m+c-\underline{S}) \int \left( 
\phi_t'' 
- 
\left|\partial \phi_t'\right|^2_{\phi_t}	\right)\omega_{\phi_t}^{[n]} 
- 
2\pi\alpha \int \phi_t'' \mathbf{ch}_2\left(h_t\right)\wedge \omega_{\phi_t}^{[n-2]}\\
& \qquad
+ \alpha \int  \tr 
\left( \frac{\sqrt{-1}}{2\pi}F_{h_t}  B_t\wedge \sqrt{-1}\partial\bar\partial \phi_t'\wedge \omega_{\phi_t}^{[n-2]}\right)\\
& \qquad\qquad 
+
2\pi \alpha \int\sqrt{-1}\partial\phi_t'\wedge\bar\partial \phi_t'\wedge \mathbf{ch}_2\left(h_t\right)\wedge \omega_{\phi_t}^{[n-3]}. 
	\end{split}
	\end{equation}
	
Differentiating the Hermitian-Einstein equation with respect to $t$, we obtain 
\[
\frac{\sqrt{-1}}{2\pi}F_{h_t}' \wedge \omega_{\phi_t}^{[n-1]}
+ 
\frac{\sqrt{-1}}{2\pi}F_{h_t}\wedge \omega_{\phi_t}^{[n-2]}\wedge \sqrt{-1}\partial\bar\partial \phi_t'
= 
\lambda \mathbf{1}_E \otimes \omega_{\phi_t}^{[n-1]}\wedge \sqrt{-1}\partial\bar\partial \phi_t'.
\]	
By multiplying both sides on the right by the endomorphism $B_t$ we obtain 

\begin{align*}
 \alpha 
 & \int  \tr 
\left( \frac{\sqrt{-1}}{2\pi}F_{h_t} B_t\wedge \sqrt{-1}\partial\bar\partial \phi_t'\wedge \omega_{\phi_t}^{[n-2]}\right)\\
& 
=
-\frac{\alpha}{2\pi} 
\int \tr\left(
\sqrt{-1}\bar\partial \partial_{h_t}B_t\; B_t\wedge \omega_{\phi_t}^{[n-1]}\right)
 + 
\alpha\lambda \int \tr \left( B_t\sqrt{-1}\partial\bar\partial \phi_t'\wedge \omega_{\phi_t}^{[n-1]}\right)\\
& 
=
-\frac{\alpha}{2\pi}\int \tr \left(\sqrt{-1}\partial_{h_t}B_t\wedge\bar\partial B_t\wedge \omega_{\phi_t}^{[n-1]}\right)
-
\alpha\lambda \int \tr\left(\sqrt{-1}\partial_{h_t}B_t\wedge \bar\partial \phi_t'\wedge \omega_{\phi_t}^{[n-1]}\right)\\
& 
= 
-\frac{\alpha}{2\pi}
\int \tr \left( 
\sqrt{-1}\left(\partial_{h_t}B_t+\pi\lambda \partial\phi_t'\otimes \mathbf{1}_E\right)\wedge 
\left(\bar\partial B_t+\pi\lambda \bar\partial\phi_t'\otimes \mathbf{1}_E\right)\wedge \omega_{\phi_t}^{[n-1]}\right)\\
&\qquad
+ 
\frac{\pi\lambda^2 m\alpha}{2} \int \sqrt{-1}\partial \phi_t'\wedge \bar\partial \phi_t'\wedge \omega_{\phi_t}^{[n-1]}\\
& 
= 
-\frac{\alpha}{2\pi}
\int \left|\partial_{h_t}B_t + \pi\lambda \partial\phi_t'\otimes \mathbf{1}_E\right|_{h_t,\omega_{\phi_t}}^2 \omega_{\phi_t}^{[n]}
+ 
\frac{\pi\lambda^2 m\alpha}{2} \int \sqrt{-1}\partial \phi_t'\wedge \bar\partial \phi_t'\wedge \omega_{\phi_t}^{[n-1]}
\end{align*}
where in the last equality we have used Lemma \ref{eqn:dB-dbarB}, and therefore this term is bounded from above. 

On the other hand, the term 	
	\begin{align*}
	2\pi\alpha \int - \phi_t'' \mathbf{ch}_2\left(h_t\right)\wedge \omega_{\phi_t}^{[n-2]} 
	& \leq 
	2\pi\alpha \int_{\mathbf{ch}_2\left(h_t\right)\wedge \omega_{\phi_t}^{[n-2]} \leq 0} -\phi_t''\mathbf{ch}_2\left(h_t\right)\wedge \omega_{\phi_t}^{n-2} \\
	&\leq 
	C \int_{\mathbf{ch}_2\left(h_t\right)\wedge \omega_{\phi_t}^{[n-2]} \leq 0} -\mathbf{ch}_2\left(h_t\right)\wedge \omega_{\phi_t}^{[n-2]} \\
	& 
	= 
	C\left(  \int -\mathbf{ch}_2\left(h_t\right)\wedge \omega_{\phi_t}^{[n-2]} 
	+ 
	 \int_{\mathbf{ch}_2\left(h_t\right)\wedge \omega_{\phi_t}^{[n-2]} \geq 0} \mathbf{ch}_2\left(h_t\right)\wedge \omega_{\phi_t}^{[n-2]}
	\right)\\
	& \leq 
	C\left( 1 + 
	 \int_{\mathbf{ch}_2\left(h_t\right)\wedge \omega_{\phi_t}^{[n-2]} \geq 0} \mathbf{ch}_1^2\left(h_t\right)\wedge \omega_{\phi_t}^{[n-2]}\right) \\
	 & 
	 = 
	 C\left( 1 + 
	 \int_{\mathbf{ch}_2\left(h_t\right)\wedge \omega_{\phi_t}^{[n-2]} \geq 0} 
	 \left( m^2\lambda^2-  \left|\mathbf{ch}_1\left(h_t\right)\right|^2\right) \omega_{\phi_t}^{[n]}\right)\\
	 & \leq 
	 C, 
	\end{align*}
	where we have used the Kobayashi-L\"ubke inequality \eqref{ineq:KL-inequality} and the uniform $C^{1,\bar 1}$ estimate $0\leq \phi_t''\leq C$ along the $\epsilon$-geodesic. 
	
	To obtain a uniform upper bound of the last term in equation \eqref{eqn:second-upper-bound} we proceed as follows. 
	\begin{equation}
	\begin{split}
	2\pi \alpha 
	&\int\sqrt{-1}\partial\phi_t'\wedge\bar\partial \phi_t'\wedge \mathbf{ch}_2\left(h_t\right)\wedge \omega_{\phi_t}^{[n-3]}\\
	& = 
	\int 2\pi \alpha \left|\partial \phi_t'\right|_{\omega_{\phi_t}}^2 \mathbf{ch}_2\left(h_t\right)\wedge \omega_{\phi_t}^{[n-2]}
	-
	\alpha
	\tr\left(
	\frac{\sqrt{-1}}{2\pi} F_{h_t}\wedge  \omega_{\phi_t}^{[n-1]}\; F_{h_t}\left(\nabla^{1,0}\phi_t',\nabla^{0,1}\phi_t'\right)
	\right) 
	+ 
	\frac{\alpha}{2\pi}
	\left|\nabla^{0,1}\phi_t'\lrcorner  F_{h_t}\right|_{h_t,\omega_{\phi_t}}^2 \omega_{\phi_t}^{[n]}\\
	& = 
	\int 2\pi \alpha \left|\partial \phi_t'\right|_{\omega_{\omega_{\phi_t}}}^2 \mathbf{ch}_2\left(h_t\right)\wedge \omega_{\phi_t}^{[n-2]}
	-
	\alpha \lambda \int
	\tr\left(
	 F_{h_t}\left(\nabla^{1,0}\phi_t',\nabla^{0,1}\phi_t'\right)
	\right) \omega_{\phi_t}^{[n]}
	+ 
	\frac{\alpha}{2\pi}
	\int
	\left|\nabla^{0,1}\phi_t'\lrcorner  F_{h_t}\right|_{h_t,\omega_{\phi_t}}^2 \omega_{\phi_t}^{[n]}.
	\end{split}
	\end{equation}

The first term is bounded from above since	
	\begin{equation*}
	\begin{split}
		\int  \left|\partial \phi_t'\right|_{\omega_{\phi_t}}^2 \mathbf{ch}_2\left(h_t\right)\wedge \omega_{\phi_t}^{[n-2]}
		& \leq 
		C\int_{\mathbf{ch}_2\left(h_t\right)\wedge \omega_{\phi_t}^{[n-2]}\geq 0} \mathbf{ch}_2\left(h_t\right)\wedge \omega_{\phi_t}^{[n-2]}\\
		& \leq 
		C \int_{\mathbf{ch}_2\left(h_t\right)\wedge \omega_{\phi_t}^{[n-2]}\geq 0} \mathbf{ch}_1^2\left(h_t\right)\wedge \omega_{\phi_t}^{[n-2]}\\
		& \leq 
		C
	\end{split}
	\end{equation*}
	as above.
	
	Using normal complex coordinates for $\omega_{\phi_t}$ and normal holomorphic frame of $E$ for $h$ such that $\frac{\partial \phi_t'}{\partial z^j}=0$ for $j=2,\cdots, n$ at any fixed point, we obtain the following pointwise bounds
	\begin{itemize} 
	\item 
		\[
	m\lambda^2 
	\leq 
	\left( \frac{1}{2\pi} \right)^2\sum_\alpha \left|\sum_i F_{\alpha\bar \alpha i\bar i}\right|^2
	\leq 
	\frac{n}{\left(2\pi\right)^2}\left|F_{h_t}\right|_{h_t,\omega_{\phi_t}}^2, 
	\]
	\item 
	\[
	\left| \tr\left(
	F_{h_t}\left(\nabla^{1,0}\phi_t',\nabla^{0,1}\phi_t'\right)
	\right) \right|^2
	= 
	\left|\sum_\alpha \phi_{t,1}'\phi_{t,\bar 1}'F_{\alpha\bar\alpha 1\bar 1}\right|^2 
	\leq 
	m \left|\partial\phi_t'\right|_{\phi_t}^4\left|F_{h_t}\right|_{h_t,\omega_{\phi_t}}^2,
	\]
	\item 
	\[
	\left|\nabla^{0,1}\phi_t'\lrcorner  F_{h_t}\right|_{\omega_{\phi_t}}^2
	\leq 
	n\left|\partial\phi_t'\right|_{\omega_{\phi_t}}^2\left|F_{h_t}\right|_{h_t,\omega_{\phi_t}}^2. 
	\]
	\end{itemize}
By the uniform bound of $\left|\nabla^{0,1}\phi_t'\right|_{\omega_{\phi_t}}$ along the $\epsilon$-geodesics, and the well-known  \emph{Energy Identity} for Hermitian-Einstein metric (following from integration of Lemma \ref{lemma-HYM}), 
\begin{equation}
\int \left|F_{h_t}\right|^2_{h_t,\omega_{\phi_t}} \omega_{\phi_t}^{[n]}
= 
4\pi^2\lambda^2m \int\omega_{\phi_t}^{[n]}
- 
8\pi^2 \int \mathbf{ch}_2\left(h_t\right)\wedge \omega_{\phi_t}^{[n-2]}, 
\end{equation}
we can obtain that $g_\epsilon''(t)\leq C$ for $t\in [0,1]$ and $\epsilon\in (0,1]$. 	
\end{proof}

\subsection{Uniqueness of K\"ahler-Yang-Mills metrics via perturbation}

Let $\left(\phi_0,h_0\right)$ and $\left(\phi_1, h_1\right)$ be two solutions to the K\"ahler-Yang-Mills equations. And let $\left(\phi_{t,\epsilon},h_{t,\epsilon}\right)$ be the solution to the system \eqref{eq:epsilongeodesics} with $\phi_{0,\epsilon}=\phi_0$ and $\phi_{1,\epsilon}=\phi_1$ for any $t\in [0,1]$ and $\epsilon\in (0,1]$.

Recall the formula 
\[
    g_\epsilon'(t)
     = 
     \int \phi_{t,\epsilon}'\left(
    2\pi\alpha \lambda \mathbf{ch}_1\left(h^{\phi_{t,\epsilon}}\right) \wedge \omega_{\phi_{t,\epsilon}}^{[n-1]} 
    -
   2\pi\alpha \mathbf{ch}_2\left(h^{\phi_{t,\epsilon}}\right)\wedge \omega_{\phi_{t,\epsilon}}^{[n-2]}
   +(c-\underline{S})\omega_{\phi_{t,\epsilon}}^{[n]} \right)
   .
\]
Obviously, $g_\epsilon'(0)$ is bounded from above and below (uniformly in $\epsilon$), and $g_\epsilon(0)$ is constant in $\epsilon$. We already know that $-C\epsilon\leq g_{\epsilon}''  \leq C$. These facts  imply that after passing to a convergent subsequence $\epsilon_i \rightarrow 0$, $g_{\epsilon_i}(t)$ converges to $g(t)$ in the $C^1$ sense, and that the limiting $g(t)$ is convex on $[0,1]$. Fix this subsequence and let $\epsilon \in \left\{\epsilon_i\right\}$ for taking limits.

Suppose $\phi_0$ and $\phi_1$ are two KYM metrics, then

\begin{equation}
\begin{split}
g'(0)
=\lim_{\epsilon\to 0}
g_\epsilon'(0)
& = 
\lim_{\epsilon\to 0}
\int \frac{\partial \phi_{t,\epsilon}}{\partial t}|_{t=0}\left(
2\pi\alpha\lambda \mathbf{ch}_1\left(h^{\phi_0}\right) \wedge \omega_{\phi_0}^{[n-1]} 
-
2\pi\alpha \mathbf{ch}_2\left(h^{\phi_0}\right)\wedge\omega_{\phi_0}^{[n-2]}
+(c-\underline{S})\omega_{\phi_{0}}^{[n]}
\right)\\
& = 
\int \frac{\partial \phi_{t}}{\partial t}|_{t=0}\left(
2\pi\alpha\lambda \mathbf{ch}_1\left(h^{\phi_0}\right) \wedge \omega_{\phi_0}^{[n-1]}
-
2\pi\alpha\mathbf{ch}_2\left(h^{\phi_0}\right)\wedge\omega_{\phi_0}^{[n-2]}
+(c-\underline{S})\omega_{\phi_{0}}^{[n]}
\right)
\end{split}
\end{equation}
where $\left\{\phi_t\right\}_{t\in [0,1]}$ is the unique $C^{1,\bar 1}$ geodesic connecting $\phi_0$ and $\phi_1$. Similarly, we have 
\[
g'(1)
= 
\int \frac{\partial \phi_{t}}{\partial t}|_{t=1}\left(
2\pi\alpha\lambda \mathbf{ch}_1\left(h^{\phi_1}\right) \wedge \omega_{\phi_1}^{[n-1]} 
-
2\pi\alpha\mathbf{ch}_2\left(h^{\phi_1}\right)\wedge\omega_{\phi_1}^{[n-2]}
+(c-\underline{S})\omega_{\phi_{1}}^{[n]}
\right). 
\]
It follows from results in K\"ahler geometry \cite{BB, CLP} that $\mathcal{K}\left(\phi_t\right)$ is convex. As a consequence, we conclude that $\mathcal{K}\left(\phi_t\right)+g(t)$ is a convex (not necessarily differentiable) function on $[0,1]$ such that 
\[
\frac{\mathrm{d}}{\mathrm{d}t}|_{t=0}^+\left( \mathcal{K}\left(\phi_t\right)+g(t)\right)
\geq 
\int \frac{\partial \phi_{t}}{\partial t}|_{t=0}\left( \left(-S_{\omega_{\phi_0}}+c\right)\omega_{\phi_0}^{[n]}
+
2\pi\alpha\lambda\mathbf{ch}_1\left(h^{\phi_0}\right) \wedge \omega_{\phi_0}^{[n-1]} 
-
2\pi\alpha\mathbf{ch}_2\left(h^{\phi_0}\right)\wedge\omega_{\phi_0}^{[n-2]}
\right)=0, 
\]
and 
\[
\frac{\mathrm{d}}{\mathrm{d}t}|_{t=1}^-\left( \mathcal{K}\left(\phi_t\right)+g(t)\right)
\leq 
\int \frac{\partial \phi_{t}}{\partial t}|_{t=1}\left( \left(-S_{\omega_{\phi_1}}+c\right)\omega_{\phi_1}^{[n]}
+
2\pi\alpha\lambda \mathbf{ch}_1\left(h^{\phi_1}\right) \wedge \omega_{\phi_1}^{[n-1]}
-
2\pi\alpha \mathbf{ch}_2\left(h^{\phi_1}\right)\wedge\omega_{\phi_1}^{[n-2]}
\right)=0. 
\]

At this point, we consider a twisting $(1,1)$-form $\chi>0$, a constant $c_s=c-\frac{s}{V}\int_X\chi\wedge\omega^{[n-1]}=c-sc_\chi$ (where $V=\int_X \omega^{[n]}$), and the corresponding \emph{Twisted KYM equation}, 

\begin{align}
\left( -S_{\omega_\phi}+c_s\right)\omega_\phi^{[n]} 
+ 
2\pi\alpha\lambda \mathbf{ch}_1\left(h\right) \wedge \omega_\phi^{[n-1]} 
- 2\pi\alpha  \mathbf{ch}_2\left(h\right)\wedge \omega_\phi^{[n-2]}
+ 
s \chi\wedge  \omega_\phi^{[n-1]} 
& =
0, \nonumber \\
\frac{\sqrt{-1}}{2\pi}F_{h}\wedge \omega_\phi^{[n-1]}-\lambda \omega_\phi^{[n]}\otimes \mathbf{1}_E
& = 
0,
\label{eq:twisted-KYM}
\end{align}
which is the Euler-Lagrange equation of $\widetilde{\mathcal{K}}_\alpha\left(\phi,h\right)+s\mathcal{J}_\chi\left(\phi\right)$ where we define
$\mathcal J_\chi$ by specifying
\begin{equation}\label{eq:normalized-J-functional}
\left.d\mathcal J_\chi\right|_\phi(\dot\phi)
=
\int_X\dot\phi\left(
\chi\wedge\omega_\phi^{[n-1]}
-c_\chi\omega_\phi^{[n]}
\right).
\end{equation}
Under the assumptions of the automorphism group of $M$ being discrete and $E$ being simple, the KYM metrics $\omega_{\phi_0}$ and $\omega_{\phi_1}$ can be perturbed to \emph{Twisted KYM metrics} $\omega_{\phi_0^s}$ and $\omega_{\phi_1^s}$ for $s\in \left(-\delta,\delta\right)$ by the implicit function theorem.  Indeed, we have the following lemma.
\begin{lemma}
\label{lem:perturbation}
Fix $\alpha>0$ and a K\"ahler form $\chi$.  Let
$(\omega_*=\omega_{\phi_*},h_*)$ be a smooth solution of the
K\"ahler--Yang--Mills equations with coupling constant $\alpha$ on a
compact K\"ahler manifold $X$.
Assume that $\operatorname{Aut}(X)$ is discrete and that $E$ is simple.
Then there are $\delta>0$ and a smooth family
$(\phi_s,h_s)$, $|s|<\delta$, of solutions of
\eqref{eq:twisted-KYM}, with $(\phi_s,h_s)|_{s=0}=(\phi_*,h_*)$.
After imposing
\begin{equation}\label{eq:twisted-family-normalization}
\int_X(\phi_s-\phi_*)\omega_{\phi_*}^{[n]}=0,
\qquad
\int_X\log\det(h_sh_*^{-1})\omega_{\phi_*}^{[n]}=0,
\end{equation}
this family is locally unique.
\end{lemma}
\begin{proof}
We first prove that the Hermitian-Einstein equation can be solved to yield a Hermitian metric depending smoothly on the background K\"ahler potential. Indeed, define the normalized spaces 
\begin{align*}
\mathcal B_0^{r,\beta}
&=
\left\{\psi\in C^{r,\beta}(X,\mathbb R):
\int_X\psi\,\omega_*^{[n]}=0\right\},\\
\mathcal E_0^{r,\beta}
&=
\left\{u\in C^{r,\beta}\left(X,\operatorname{Herm}_{h_*}(\End E)\right):
\int_X\tr(u)\,\omega_*^{[n]}=0\right\}
\end{align*}
for $r\in \mathbb{N},\;\beta\in (0,1)$. For $u$ sufficiently small define a Hermitian metric by
\[
h_u(v,w)=h_*\left(e^u v,w\right)
\]
where $e^u=\exp u$ and $\exp: \operatorname{Herm}_{h_*}(\End E)\rightarrow \operatorname{Herm}_{h_*}^+(\End E)$ refers to the restriction of the standard expoential map. This is a smooth coordinate chart on the space of Hermitian metrics near
$h_*$.  In this chart the second normalization required in the statement of the Lemma is precisely
$\int_X\tr(u)\omega_*^{[n]}=0$.

Let
\[
\rho_\psi=\frac{\omega_\psi^{[n]}}{\omega_*^{[n]}}
\]
and consider
\begin{align*}
\mathcal H:\mathcal U\times\mathcal E_0^{k+4,\beta}
&\longrightarrow \mathcal E_0^{k+2,\beta},\\
\mathcal H(\psi,u)
&=
\rho_\psi e^{u/2}
\left(
\frac{\sqrt{-1}}{2\pi}\Lambda_{\omega_\psi}F_{h_u}
-\lambda\mathbf 1_E
\right)e^{-u/2},
\end{align*}
where $\mathcal U^{k+4,\beta}$ is any fixed neighbourhood of $0$ in
$\mathcal B_0^{k+4,\beta}$ on which $\omega_\psi>0$.  Conjugation by
$e^{u/2}$ identifies endomorphisms self-adjoint with respect to $h_u$ with
endomorphisms self-adjoint with respect to the fixed Hermitian metric $h_*$. Standard properties show that $\mathcal H$ is a smooth map. It is easy to see that its linearization has trivial kernel and hence by the implicit function theorem, we get a unique smooth map $u=\mathcal F_{k,\beta}(\psi)$ solving the Hermitian-Einstein equation, defined on an open neighborhood $\mathcal{W}^{k+4,\beta}$ of $0$. The maps for different $k$ are compatible with each other by uniqueness of solutions, under the assumption that $E$ is simple and the Hermitian-Einstein metrics are normalized, for the Hermitian-Einstein equation.\\

\indent Now we proceed to prove Lemma \ref{lem:perturbation}. Given a K\"ahler potential $\phi$, define $\psi:=\phi-\phi_*$.  Apply the smooth Hermitian-Einstein reduction above to obtain
the normalized Hermitian-Einstein metric $h^\psi=h_{\mathcal F_{k,\beta}(\psi)}$ for every sufficiently
small normalized $\psi$. For a fixed $0<\beta<1$, let
\begin{align*}
\mathcal B_0^{4,\beta}
=
\left\{\psi\in C^{4,\beta}(X,\mathbb R):
\int_X\psi\,\omega_*^{[n]}=0\right\},
\qquad
C_0^{0,\beta}
=
\left\{f\in C^{0,\beta}(X,\mathbb R):
\int_Xf\,\omega_*^{[n]}=0\right\}.
\end{align*}
Define the reduced map $\mathcal N: \mathcal{W}^{4,\beta} \rightarrow C_0^{0,\beta}$, which is nonlocal, by identifying a top-degree form with
a function using the fixed volume form $\omega_*^{[n]}$:
\begin{align}
\mathcal N(\psi,s)
=\frac{1}{\omega_*^{[n]}}\Big[&
\left(-S_{\omega_\psi}+c_s\right)\omega_\psi^{[n]}
+2\pi\alpha\lambda\mathbf{ch}_1(h^\psi)
\wedge\omega_\psi^{[n-1]}\nonumber\\
&-2\pi\alpha\mathbf{ch}_2(h^\psi)
\wedge\omega_\psi^{[n-2]}
+s\chi\wedge\omega_\psi^{[n-1]}
\Big].
\label{eq:reduced-twisted-map}
\end{align}
 The smooth Hermitian-Einstein reduction above, the
smoothness of the scalar-curvature operator
$C^{4,\beta}\to C^{0,\beta}$, and the fact that the Chern forms depend smoothly on the second jet of a Hermitian metric show that $\mathcal N$ is
a smooth map of Banach spaces.  Thus the non-locality of $h^\psi$ causes no loss of differentiability. 
We next prove that
\[
L:=D_\psi\mathcal N|_{(0,0)}:
\mathcal B_0^{4,\beta}\longrightarrow C_0^{0,\beta}
\]
is an isomorphism.  Its fourth-order principal part is $\Delta_{\omega_*}^2$,
coming from the familiar variation formula about $-S_{\omega_\psi}$ with respect to the K\"ahler potential. It is easy to see that the inverse of the Hermitian-Einstein operator (that is, the function $\mathcal F_{0,\beta}$) gives a bounded map
$\eta\mapsto B_\eta$ from $C^{4,\beta}$ to $C^{4,\beta}$.  Its contribution
to the variation of the Chern forms contains at most two derivatives of
$B_\eta$ and hence takes values in $C^{2,\beta}$.  All the remaining terms
in $L-\Delta_{\omega_*}^2$ likewise contain at most two derivatives of
$\eta$.  Since the inclusion $C^{2,\beta}\hookrightarrow C^{0,\beta}$ is
compact on $X$, we have
\[
L=\Delta_{\omega_*}^2+K
\]
with $K:\mathcal B_0^{4,\beta}\to C_0^{0,\beta}$ compact.  The bi-Laplacian
is an isomorphism between these mean-zero spaces.  It follows that $L$ is
Fredholm of index zero. It remains to prove that $L$ has trivial kernel.  Let $\eta\in\ker L$ and
consider the reduced functional
$\mathcal K_\alpha(\psi)=\widetilde{\mathcal K}_\alpha(\psi,h^\psi)$.
At a solution, the proof of the second-variation formula in Theorem
\ref{thm:second-variation} gives
\begin{align}
0
=\int_X\eta L\eta\,\omega_*^{[n]}
=&\int_X\left|\bar\partial\left(
\nabla_{\omega_*}^{1,0}\eta\right)\right|_{\omega_*}^2
\omega_*^{[n]}
+\frac{\alpha}{2\pi}
\int_X\left|
\partial_{h_*}B_\eta
-\nabla_{\omega_*}^{0,1}\eta\mathbin{\lrcorner}F_{h_*}
\right|_{h_*,\omega_*}^2\omega_*^{[n]}.
\label{eq:reduced-linearization-kernel}
\end{align}
In particular,
$\bar\partial\left(\nabla_{\omega_*}^{1,0}\eta\right)=0$, so
$\nabla_{\omega_*}^{1,0}\eta$ is a holomorphic vector field.  The
discreteness of $\operatorname{Aut}(X)$ implies that this vector field is
zero.  Thus $\eta$ is constant, and the normalization condition forces $\eta=0$.
Since $L$ is Fredholm of index zero, it is therefore an isomorphism. The implicit function theorem finishes the proof.

\end{proof}
The function $\mathcal{J}_\chi(\phi_t)$ is strictly convex and we therefore obtain that $\mathcal{K}\left(\phi_t\right)+s\mathcal{J}_\chi(\phi_t)+g(t)$ (for $s>0$) is a strictly convex function on $[0,1]$ such that the right derivative at $0$ and the left derivative at $1$ are both $0$. This implies the perturbed \emph{Twisted K\"ahler-Yang-Mills metrics} $\omega_{\phi_0^s}=\omega_{\phi_1^s}$ for $s>0$ and therefore the original unperturbed K\"ahler-Yang-Mills metrics $\omega_{\phi_0}$ and $\omega_{\phi_1}$ coincide, establishing Theorem \ref{thm:uniqueness}.

\begin{remark}
	This argument avoids the discussion of weak Hermitian-Einstein metrics, by exploiting the ``symplectic by reduction'' structure of the K\"ahler-Yang-Mills equations. 
\end{remark}
The arguments above also prove the following result.
\begin{theorem}
	If there exists a K\"ahler-Yang-Mills metric, then $\mathcal{K}_\alpha$ is bounded from below on $\mathcal{H}_\omega$, and $\widetilde{\mathcal{K}}_\alpha$ is bounded from below on $\widetilde{\mathcal{P}}$. 
\end{theorem}

Inspired by the general moment map framework and the succesful characterization of the relation between the formal \emph{Kempf-Ness functional} and canonical metrics \cite{Tian, donald, HJ, CC}, we make the following conjecture:
\begin{conjecture}
	Assume that $X$ has discrete automorphism group, then the reduced $\alpha$-K-energy is proper if and only if there exists a K\"ahler-Yang-Mills metric. 
\end{conjecture}

\subsection{Explicit formula}
In this subsection we integrate out to obtain the following explicit formula of the $\alpha$-K-energy $\widetilde{\mathcal{K}}_\alpha$ for the sake of completeness and self-containedness. 

\begin{lemma}
	Let $\mathcal{K}$ be a slightly modified K-energy depending on a constant $c$ which is usually the average scalar curvature $\bar{S}$ (the derivative of $\mathcal{K}$ is $-\int_X \phi^{'} (S-c) \omega_{\phi}^n$). Define the $\alpha$-K-energy $\widetilde{\mathcal{K}}_\alpha$ on $\widetilde{\mathcal{P}}$ as follows.
	\begin{align}
	n!\widetilde{\mathcal{K}}_\alpha(\phi, H)
	&  = \mathcal{K}(\phi) +  \alpha  \int_ X \mathbf{bc}_2 \left(h^H,h\right) \wedge n\omega_{\phi}^{n-1} -\alpha\lambda \int _{X} \mathbf{bc}_1 \left(h^H,h\right) \omega_{\phi}^n  \nonumber \\
	& \qquad +  2\pi \alpha \lambda \sum_{r=0}^{n-1}  \binom{n}{r+1} \int _X  \phi\omega ^{n-1-r} \wedge\left(\sqrt{-1}\partial\bar\partial\phi\right)^r\wedge \mathbf{ch}_1\left(h\right) \nonumber \\ 
	& \qquad -  2 \pi \alpha  \sum _{r=0}^{n-2}  n\binom{n-1}{r+1} \int_X  \phi \omega ^{n-2-r}\wedge \left(\sqrt{-1}\partial\bar\partial \phi\right)^r\wedge \mathbf{ch}_2\left(h\right).
	\label{newmabuchi}
	\end{align}
	
\end{lemma}
\begin{proof}
	The following simple calculation demonstrates the result.
	\begin{equation}
	\begin{split} 
	n!\frac{d\widetilde{\mathcal{K}}_\alpha}{dt}  
	& = \frac{d\mathcal{K}}{dt} +  \alpha n \int_ X \frac{d \mathbf{bc}_2 \left(h^H,h\right)}{dt}\wedge \omega_{\phi}^{n-1} -\alpha \lambda\int _{X} \frac{d \mathbf{bc}_1 \left(h^H,h\right)}{dt} \omega_{\phi}^n   \\
	&\quad + \alpha n \int_ X \mathbf{bc}_2 \left(h^H,h\right)\wedge (n-1) \omega_{\phi}^{n-2} \spbp \phi^{'}\nonumber
	-\alpha\lambda \int _{X} \mathbf{bc}_1 \left(h^H,h\right) n\omega_{\phi}^{n-1}\wedge\spbp \phi^{'}\\
	&\quad +  2\pi \alpha \sum_{r=0}^{n-1} \int _X \lambda (r+1) \binom{n}{r+1} \phi^{'} \omega ^{n-1-r}\wedge  \left(\spbp \phi\right)^{r}\wedge \mathbf{ch}_1\left(h\right)  \\
	&\quad -  2 \pi \alpha n \sum _{r=0}^{n-2} \int _X  (r+1)\binom{n-1}{r+1} \phi'  \omega ^{n-2-r} \wedge\left(\spbp \phi\right)^{r} \wedge\mathbf{ch}_2\left(h\right)\\
	& = -  \int _X \phi^{'} (S-c) \omega_{\phi}^n + \displaystyle \alpha n \int_ X \frac{d \mathbf{bc}_2 \left(h^H,h\right)}{dt} \wedge \omega_{\phi}^{n-1} -\alpha \lambda \int _{X} \frac{d \mathbf{bc}_1 \left(h^H,h\right)}{dt} \omega_{\phi}^n \\
	&\quad -2\pi  \alpha n(n-1) \int_ X \phi^{'} \left( \mathbf{ch}_2 \left( h^H \right)-\mathbf{ch}_2\left(h\right) \right)\wedge \omega_{\phi}^{n-2} 
	+2\pi\alpha n \lambda \int _{X} \phi^{'} \left( \mathbf{ch}_1\left(h^H\right)-\mathbf{ch}_1\left(h\right)\right) \wedge\omega_{\phi}^{n-1}\\
	& \quad + 2\pi \alpha \lambda n \int _X  \phi^{'}\omega_{\phi}^{n-1} \wedge \mathbf{ch}_1\left(h\right)  - 2 \pi \alpha n(n-1)\int _X\phi^{'} \omega_{\phi}^{n-2}\wedge \mathbf{ch}_2\left(h\right) .
	\label{calc}
	\end{split}
	\end{equation}
	At this point, we may use the calculation in \cite{donald} to complete the proof. However, we will repeat this calculation for convenience of the reader.  Let $h(t)$ be a smooth path of Hermitian metrics, we have the following. Using $\dbar F=0$, we get
	\begin{equation}
	\begin{split}
	\frac{d \mathbf{ch}_k(h)}{dt} 
	& = \frac{\sqrt{-1}}{(k-1)!2\pi} \tr \left ( \left (\frac{\sqrt{-1}}{2\pi}F \right)^{k-1} \wedge\dbar \left (\partial h h^{-1} \right) ^{'}\right) \\
	& =_{\color{red}\bar\partial F_\alpha^{\phantom{\alpha}\beta}=0} \frac{\sqrt{-1}}{(k-1)!2\pi} \dbar \tr \left ( \left (\frac{\sqrt{-1}}{2\pi}F \right)^{k-1} \wedge\left( \partial h ^{'} h^{-1} -\partial h h^{-1} h^{'} h^{-1}\right) \right) \\
	& = \frac{\sqrt{-1}}{(k-1)!2\pi} \dbar \tr \left ( \left (\frac{\sqrt{-1}}{2\pi}F \right)^{k-1} \wedge\left( \partial \left( h ^{'} h^{-1} \right) +\left[ h^{'} h^{-1},\partial h h^{-1} \right] \right) \right).
	\label{bcexp}
	\end{split}
	\end{equation}

	Denote $A= \partial hh^{-1}$ and $B=h'h^{-1}$, using the Bianchi identity $\partial F + \left[F, A\right]=0$, we obtain 
	\begin{equation*}
	\begin{split}
	F^{k-1} \wedge\left( \partial B +\left[ B,A \right]\right)
	& = 
	\partial \left( F^{k-1}B\right) 
	+ 
	F^{k-1}\wedge AB - A\wedge F^{k-1} B
	+ 
	F^{k-1}\wedge\left[B,A\right]\\
	& = 
	\partial \left(F^{k-1}B\right) 
	+ 
	\left[ F^{k-1}B, A\right]
	\end{split}
	\end{equation*}
	and it follows that 
	\[
	\tr \left( F^{k-1}\wedge \left( \partial B +\left[ B,A \right]\right) \right)
	= 
	\partial \tr \left(F^{k-1}B\right).
	\]
	As a consequence, 
	\begin{equation}
	\begin{split}
	\dbar\partial \frac{d \mathbf{bc}_k \left(h^H,h\right)}{dt}
	& = 
	\frac{1}{(k-1)!}
	\dbar \partial \tr \left ( \left (\frac{\sqrt{-1}}{2\pi}F \right)^{k-1} h ^{'} h^{-1}  \right).
	\end{split}
	\end{equation}
	The final formula follows.
	
\end{proof}

\begin{remark}
	In a local holomorphic frame $\left\{ \mathbf{e}_\alpha\right\}_{\alpha=1,\cdots, m}$ of $E$, the metric $h$ is represented by $h_{\alpha\bar\beta}$ and $h'h^{-1}$ represents the endomorphism $h'_{\alpha\bar\gamma}h^{\beta\bar\gamma}$, i.e. 
	\[
	h\left( \mathbf{e}_\alpha, \mathbf{e}_\beta\right) = h_{\alpha\bar \beta}, \;\; 
	h'h^{-1}\left(\mathbf{e}_\alpha\right) = 
	h'_{\alpha\bar\gamma}h^{\beta\bar \gamma}\mathbf{e}_\beta.
	\]
	Similarly  the Chern connection matrix $1$-form $\partial hh^{-1}$ is represented by $\partial h_{\alpha\bar\gamma}h^{\beta\bar\gamma}$, i.e. 
	\[
	\nabla \mathbf{e}_\alpha 
	= 
	\partial h_{\alpha\bar\gamma}h^{\beta\bar\gamma}\otimes \mathbf{e}_\beta.
	\]
	The curvature endomorphism $F=\bar\partial \left(\partial hh^{-1}\right)= \bar\partial \partial hh^{-1} + \partial hh^{-1}\wedge \bar\partial hh^{-1}$ takes the form 
	\[
	F(\mathbf{e}_\alpha)
	=
	F_\alpha^{\phantom{\alpha}\beta}\otimes \mathbf{e}_\beta
	=
	h^{\beta\bar\gamma}\left( - \partial\bar\partial h_{\alpha\bar\gamma} 
	+ 
	h^{\lambda\bar\mu}\partial h_{\alpha\bar\mu}\wedge\bar\partial h_{\lambda\bar\gamma}\right)\otimes \mathbf{e}_\beta.
	\]

\noindent Also, the endomorphism $H$ means $H\left(\mathbf e_\alpha\right)=H_\alpha^{\phantom{\alpha}\beta}\mathbf e_\beta$ and $h^H\left(\mathbf e_\alpha,\mathbf e_\beta\right)=H_\alpha^{\phantom{\alpha}\gamma}h_{\gamma\bar \beta}$. Moreover, $B_t$ represents the endomorphism that $B_t\mathbf e_\alpha=\left(H_t'\right)_\alpha^{\phantom{\alpha}\beta}\left(H_t^{-1}\right)_\beta^{\phantom{\beta}\gamma}\mathbf e_\gamma$.
\end{remark}

\subsection{A coupled J-equation} Inspired by Chen's introduction of the J-equation \cite{XX1} during his study of lower bounds of the $K$-energy, we propose the following system as a tool to study the lower bound of the $\alpha$-K-energy and the reduced $\alpha$-K-energy. Let $\chi$ be a closed real $(1,1)$-form that is positive on $X$, we can consider the system of \emph{ $J_\chi^{YM}$-equations} for the pair $\left(\phi,h\right)$: 
\begin{equation}
\begin{split}
&\frac{\sqrt{-1}}{2\pi}F_{h}\wedge \omega _{\phi} ^{[n-1]}  = \lambda \omega _{\phi} ^{[n]} \otimes \mathbf{1}_E,\\
&\chi\wedge \omega_\phi^{[n-1]}
    + 
    2\pi\alpha\lambda \mathbf{ch}_1\left(h\right)\wedge \omega_\phi^{[n-1]}
    -
    2\pi\alpha \mathbf{ch}_2\left(h\right)\wedge\omega_\phi^{[n-2]}
    = 
    c'\omega_\phi^{[n]}.
\label{Chi-Yang-Mills-sys}
\end{split}
\end{equation}

If we assume the stability of the bundle $E$ with respect to the polarization $[\omega]$, then this system of PDEs about $\left(\phi, h\right)$ is equivalent to the following {\bfseries{``non-local'' PDE}}:

\begin{equation}
\label{nonlocal-Chi-Yang-Mills-sys}
    \chi\wedge \omega_\phi^{[n-1]}
    + 
    2\pi\alpha\lambda \mathbf{ch}_1\left(h^\phi\right)\wedge \omega_\phi^{[n-1]}
    -
    2\pi\alpha \mathbf{ch}_2\left(h^\phi\right)\wedge\omega_\phi^{[n-2]}
    = 
    c'\omega_\phi^{[n]}
\end{equation}
where $h^\phi$ is a Hermitian-Einstein metric on $E$ with respect to the background K\"ahler metric $\omega_\phi$, and $c'$ is a real constant. This constant is topologically determined if the manifold is further assumed compact, which is 
\begin{align*}
    c'
    =
    c_\chi +\frac{ 2\pi\alpha \mathbf{ch}_1\cup \omega^{[n-1]} - 2\pi\alpha \mathbf{ch}_2\cup \omega^{[n-2]}}{V}.
\end{align*}

By Proposition \ref{1-form-closed} and the well-definedness of the functional 
\begin{align*}
    \int_0^1 \int_X \phi_t'\chi\wedge \omega_{\phi_t}^{[n-1]}, 
\end{align*}
there is a $\widetilde{\mathcal{J}}_\chi^{YM}$-functional defined on $\widetilde{\mathcal{P}}$ whose derivative along a smooth curve $\left(\phi_t,h_t\right)$ is 
\begin{equation}
\begin{split}
\frac{\mathrm{d}\widetilde{\mathcal{J}}_\chi^{YM}}{\mathrm{d}t}\left( \phi_t,h_t\right)
& = 
\int_X \phi_t'\left( \chi\wedge \omega_{\phi_t}^{[n-1]}
+ 
2\pi\alpha\lambda \mathbf{ch}_1\left(h_t\right) \wedge \omega_{\phi_t}^{[n-1]} 
-
2\pi\alpha\mathbf{ch}_2\left(h_t\right)\wedge\omega_{\phi_t}^{[n-2]}
-c'\omega_{\phi_t}^{[n]}
\right)\\
& + 
\alpha\int_X \tr \left( \left( \frac{\sqrt{-1}}{2\pi}F_{h_t}\wedge \omega_{\phi_t}^{[n-1]}-\lambda \omega_{\phi_t}^{[n]}\otimes \mathbf{1}_E\right) h_t'h_t^{-1}\right),
\end{split}
\end{equation}
and the critical points of $\widetilde{\mathcal{J}}_\chi^{YM}$ on $\widetilde{\mathcal{P}}$ occur precisely when solutions to the Equation \eqref{Chi-Yang-Mills-sys} exist.\\

Akin to the case of K\"ahler-Yang-Mills equations for which we consider the \emph{reduced $\alpha$-K-energy} instead of the \emph{$\alpha$-K-energy}, we can consider the reduced functional $\mathcal{J}_\chi^{YM}$ instead of $\widetilde{\mathcal{J}}_\chi^{YM}$. As the restriction of $\widetilde{\mathcal{J}}_\chi^{YM}$ to $\mathcal{P}_{HE}$, the reduced functional $\mathcal{J}_\chi^{YM}$ has the first derivative
\begin{equation}
    \frac{\mathrm{d} \mathcal{J}_\chi^{YM}}{\mathrm{d}t}\left(\phi_t\right)
    = 
    \int_X \phi_t'\left( 
    \chi\wedge \omega_{\phi_t}^{[n-1]}
    + 
    2\pi\alpha\lambda \mathbf{ch}_1\left(h^{\phi_t}\right)\wedge \omega_{\phi_t}^{[n-1]}
    -
    2\pi\alpha \mathbf{ch}_2\left(h^{\phi_t}\right)\wedge\omega_{\phi_t}^{[n-2]}
    -
    c'\omega_{\phi_t}^{[n]}
    \right).
\end{equation}
and the second derivative
\begin{equation}
\begin{split}
   \frac{\mathrm{d}^2 \mathcal{J}_\chi^{YM}}{\mathrm{d}t^2}\left(\phi_t\right)
   =
	& \int_X \left(\phi_t''-\left|\partial \phi'_t\right|_{\omega_{\phi_t}}^2\right) \left(
	\chi\wedge\omega_{\phi_t}^{[n-1]}-c'\omega_{\phi_t}^{[n]} 
	+ 
	2\pi\alpha\lambda \mathbf{ch}_1\left(h^{\phi_t}\right)\wedge\omega_{\phi_t}^{[n-1]}
 -
 2\pi \alpha \mathbf{ch}_2\left(h^{\phi_t}\right)\wedge \omega_{\phi_t}^{[n-2]}
	\right)\\
	& + 
    \int_X -\sqrt{-1}\chi\left(\nabla^{1,0}\phi_t', \nabla^{0,1}\phi_t'\right)\omega_{\phi_t}^{[n]}
	+
	\frac{\alpha}{2\pi}\int_X\left| \partial_{h^{\phi_t}} \left(\frac{\partial h^{\phi_t}}{\partial t}\left(h^{\phi_t}\right)^{-1}\right)-\nabla^{0,1}\phi_t'\lrcorner F_{h^{\phi_t}}\right|_{h^{\phi_t},\omega_{\phi_t}}^2\omega_{\phi_t}^{[n]}
	\end{split}
    \label{J-chi-YM}
	\end{equation}
along any smooth path $\{\phi_t\}_{t\in [0,1]}$. Notice that, if $\chi=\sqrt{-1}\chi_{i\bar j}\mathrm{d}z^i\wedge \mathrm{d}\bar z^j$ and $\omega_{\phi_t}=\sqrt{-1}\left(g_{\phi_t}\right)_{i\bar j}\mathrm{d}z^i\wedge \mathrm{d}\bar z^j$ in local complex coordinates $\left\{z^1, \cdots, z^n\right\}$, then
\[
-\sqrt{-1}\chi\left(\nabla^{1,0}\phi_t', \nabla^{0,1}\phi_t'\right)
= 
g_{\phi_t}^{i\bar l}g_{\phi_t}^{k\bar j}\chi_{i\bar j}\frac{\partial\phi_t'}{\partial z^k} \frac{\partial \phi_t'}{\partial \bar z^l}
\]
is nonnegative pointwise, and vanishes precisely at the points where $\partial \phi_t'=0$.\\

Evaluating this functional on the approximate geodesic segment $\left\{ \phi_{t,\epsilon}\right\}_{t\in [0,1]}$ (with $\epsilon>0$) 
 we arrive at the following expression:
\begin{equation}
    \mathcal{J}_\chi^{YM}\left(\phi_{t,\epsilon}\right)
    =
    \mathcal{J}_\chi\left(\phi_{t,\epsilon}\right)+g_\epsilon(t)+\underline S \mathcal{I}(\phi_{t,\epsilon}),
\end{equation}
where 
\begin{align*}
    \mathcal{J}_\chi\left(\phi\right)
    =
    \int_0^1 \mathrm{d}t\int_X \phi_t'\left(\chi\wedge\omega_{\phi_t}^{[n-1]}-c_\chi\omega_{\phi_t}^{[n]}\right)
\end{align*}
is the functional integrating \eqref{eq:normalized-J-functional}, $g_\epsilon(t)=\mathcal{M}_\alpha(\phi_{t,\epsilon})$ as in section \ref{sec:Newgeo}, and $\mathcal{I}$ is the usual Aubin-Yau functional defined as
\begin{align*}
    \mathcal{I}(\phi)
    = 
    \int_0^1 \mathrm{d}t\int_X \phi_t'\omega_{\phi_t}^{[n]}.
\end{align*}
Therefore, its subsequential limit (as $\epsilon\to 0$)
\[
\mathcal{J}_\chi\left(\phi_t\right)+g(t)+\underline S \mathcal{I}(\phi_t)
\]
is a strictly convex function on $[0,1]$ (the strictness comes from the strict convexity of $\mathcal{J}_\chi$). Moreover, the right derivative at $t=0$ is nonnegative while the left derivative at $t=1$ is nonpositive as above if the endpoints $\phi_0$ and $\phi_1$ are solutions to Equation \eqref{nonlocal-Chi-Yang-Mills-sys}, and therefore it is constant function on $[0,1]$. As a consequence, we conclude

\begin{prop}~
  \begin{enumerate}
    \item 
    If $\chi>0$ is a closed real $(1,1)$-form, then the functional $\mathcal{J}_\chi^{YM}$ is convex along the $C^{1,\bar 1}$-geodesic connecting any two potential in $\mathcal{H}_\omega$. In particular, smooth solution to the $J_\chi^{YM}$-equation is unique, and the existence of a smooth solution to $J_\chi^{YM}$-equation implies this functional is bounded from below on $\mathcal{H}_\omega$. 
    \item
    If $\omega$ is a K\"ahler metric such that $-\text{Ric }\omega>0$, and $\mathcal{J}_{-\text{Ric }\omega}^{YM}$ is bounded from below, then the $\alpha$-K-energy is bounded from below on $\mathcal{H}_\omega$.
    \end{enumerate}
\end{prop}

Using  Lemma \ref{lemma-HYM} and normal coordinates determined by $\chi$ at $p$ for which $\omega_\phi = \sqrt{-1}\gamma_i dz^i\wedge d\bar z^i$, we see that the Equation \eqref{nonlocal-Chi-Yang-Mills-sys} is translated as 
\begin{align*}
    \sum_{i=1}^n \frac{1}{\gamma_i}
    = c' -\pi \alpha m\lambda^2 - \frac{\alpha}{4\pi}\left|F_{h^\phi}\right|_{h^\phi, \omega_\phi}^2.
\end{align*}
Using the inequality\footnote{Equality holds if and only if $F_{h^\phi}= -\frac{2\pi\sqrt{-1}}{n}\lambda \mathbf{1}_E\otimes \omega_\phi$. } 
\begin{align*}
    \left|F_{h^\phi}\right|_{h^\phi, \omega_\phi}^2 
    \geq 
    \frac{1}{n}\sum_{\alpha,\beta}\left|\sum_{i=1}^n F_{\alpha\phantom{\beta}i\bar i}^{\phantom{\alpha}\beta}\right|^2
    = 
    4\pi^2\lambda^2\frac{m}{n},
\end{align*}
we obtain the inequality 
\begin{align*}
    \sum_{j\neq i}^n \frac{1}{\gamma_j}
    < 
    c'-\pi\alpha m\lambda^2 -\pi\alpha\frac{m}{n}\lambda^2, \;\; \forall i=1,2,\cdots, n.
\end{align*}
\begin{prop}
An analytic necessary condition for the existence of solution (in $[\omega]$) to Equation \eqref{Chi-Yang-Mills-sys} is the existence of $\widetilde\omega\in [\omega]$ such that 
\begin{align*}
    \left(c'-\pi\alpha m\left(1+\frac{1}{n}\right)\lambda^2\right)\widetilde{\omega}^{[n-1]}- \chi \wedge \widetilde{\omega}^{[n-2]}>0
\end{align*}
pointwise on $X$.
\end{prop}

More generally than above, the family of analytic conditions
\begin{align}
    \left(c'-\pi\alpha m\left(1+\frac{1}{n}\right)\lambda^2\right)\widetilde{\omega}^{[p]}- \chi \wedge \widetilde{\omega}^{[p-1]}>0, \; p=n-1, \cdots, 1, 
    \end{align}
    are all necessary, and have decreasing strength according to the decreasing of $p$ from $n-1$ to $1$.

\begin{prop}
    A numerical necessary condition for the existence of solution (in $[\omega]$) to Equation \eqref{Chi-Yang-Mills-sys} is 
    \begin{align*}
        \left(c'-\pi\alpha m\left(1+\frac{1}{n}\right)\lambda^2\right)
        \int_{Z^p} \omega^{[p]}
        - 
        \int_{Z^p}\chi \wedge \omega^{[p-1]}>0
    \end{align*}
    for any $p$-dimensional subvariety $Z^p\subset X$, $p=n-1,\cdots, 1$.
\end{prop}

We suspect that more necessary conditions are needed and hence are wary of formulating a conjecture prematurely. It is not hard to come up with a continuity path and prove openness. A more detailed study of this system would require new techniques to handle a priori estimates for fully nonlinear PDE systems (especially second and higher-order ones). The usual Aubin-Yau \cite{Aubin, Yau} method of the maximum principle and Moser iteration \cite{ChenHe} do not seem to work. Indeed, for systems of nonlinear PDEs (not this specific one), counterexamples to regularity are known \cite{Necas}.

\end{document}